\documentclass[12pt, leqno]{article}
\usepackage{amstext,amssymb,amsmath,amsbsy}
\usepackage{amscd}
\usepackage{amsfonts,esint}
\usepackage{indentfirst}
\usepackage{verbatim}
\usepackage{amsmath}
\usepackage{amsthm}
\usepackage{enumerate}
\usepackage{graphicx}
\usepackage{subfig}
\usepackage{color}
\usepackage[OT1]{fontenc}
\usepackage[latin1]{inputenc}
\usepackage[english]{babel}
\usepackage{amssymb}
\makeatletter
\DeclareRobustCommand\widecheck[1]{
	{\mathpalette\@widecheck{#1}}}
\def\@widecheck#1#2{
		\setbox\z@\hbox{\m@th$#1#2$}
	\setbox\tw@\hbox{\m@th$#1%
		\widehat{%
			\vrule\@width\z@\@height\ht\z@
			\vrule\@height\z@\@width\wd\z@}$}%
	\dp\tw@-\ht\z@
	\@tempdima\ht\z@ \advance\@tempdima2\ht\tw@ \divide\@tempdima\thr@@
	\setbox\tw@\hbox{%
		\raise\@tempdima\hbox{\scalebox{1}[-1]{\lower\@tempdima\box
				\tw@}}}%
	{\ooalign{\box\tw@ \cr \box\z@}}}
\makeatother
\newtheorem{theorem}{Theorem}
\newtheorem{lemma}{Lemma}

\newtheorem{proposition}{Proposition}
\newtheorem{corollary}{Corollary}
\newtheorem{remark}{Remark}

\newtheorem{definition}{Definition}
\numberwithin{equation}{section}

\usepackage{dsfont}

\newcommand{\supp}{\operatorname{supp}}

\newcommand{\dive}{\operatorname{div}}

\newcommand{\mR}{\mathbb{R}}
\newcommand{\ren}{\mathbb{R}^N}

\title{\bf Porous medium equation with nonlocal \\[6pt] pressure in a bounded domain}

\author{{\bf Quoc-Hung Nguyen\thanks{ E-mail address: quochung.nguyen@sns.it, Scuola Normale Superiore, Centro Ennio de Giorgi, Piazza dei Cavalieri 3, I-56100
			Pisa, Italy.}}~~\\ and \\
		{\bf Juan Luis V\'azquez\thanks{ E-mail address: juanluis.vazquez@uam.es, Departamento de Matem\'aticas, Universidad Aut\'onoma de
Madrid,  28049 Madrid, Spain.}}}

\begin{document}

\date{ }

\maketitle

\begin{abstract}
 We study a quite general family of nonlinear evolution equations of diffusive type with nonlocal effects. More precisely, we study porous medium equations with a fractional Laplacian pressure, and the problem is posed on a bounded space domain. We prove existence of weak solutions and suitable a priori bounds and regularity estimates.

\end{abstract}


\footnotesize
\tableofcontents
\normalsize

\

\section{Introduction}\label{sec.intro}	

Let $N \ge 1$  and $\Omega$ be a  bounded  open subset of $\mR^N$ with smooth boundary $\partial \Omega$. In this paper we study the following family of nonlinear evolution equations of diffusive type with nonlocal effects
\begin{equation}\label{pro1}
\left\{
\begin{array}
{ll}%
\partial_t u-\operatorname{div}\left(|u|^{m_1}\,\nabla  (-\Delta)^{-s}  (|u|^{m_2-1}u)\right)=f~~&\text{in }\Omega_T=\Omega\times (0,T),\\[6pt]
u(x,t)=0 &\text{on}
~~\partial\Omega\times (0,T),
\\[6pt]
u(x,0)=u_0(x) &\text{in}
~~\Omega,
\\
\end{array}
\right.
\end{equation}
where $u_0, f$ are bounded functions or bounded Radon measures in $\Omega$ and $\Omega_T$  respectively, and $m_1, m_2>0$. The symbol $(-\Delta)^{-s}$ with $0<s<1$ denotes the inverse of the {\sl spectral fractional Laplacian operator } with zero Dirichlet outer conditions, which is defined as follows: we denote by $-\Delta$ the Laplacian operator with homogeneous Dirichlet boundary conditions on $\Omega$. Its $L^2(\Omega)$-normalized eigenfunctions
are denoted $\varphi_j$, and its eigenvalues counted with their multiplicities are denoted $\lambda_j$:  \ $-\Delta \varphi_j=\lambda_j\varphi_j$. It is well known that
$$
0<\lambda_1\leq...\leq\lambda_j\leq...,\qquad \mbox{and } \lambda_j \asymp j^{2/N},
$$
and that $-\Delta$ is a positive self-adjoint operator in $L^2(\Omega)$ with
domain $D(-\Delta) = H^2(\Omega)\cap H^1_0(\Omega)$. The ground state $\varphi_1$ is positive and $\varphi_1(x)\asymp d(x)$ for all $x\in \Omega$, where $d(x)$ denotes distance from $x$ to boundary $\partial\Omega$. For all $0<s<1$  we define the  spectral fractional Laplacian  $(-\Delta)^{s}$ by
\begin{equation}
(-\Delta)^{s}f=\sum_{j=1}^{\infty}\lambda_j^{s}f_j \varphi_j,
\qquad f_j=\int_{\Omega}f(x)\varphi_j(x)dx\,.
\end{equation}
This formula is equivalent to the semigroup formula
\begin{equation}
(-\Delta)^{s}f= \frac1{\Gamma(-s)}\int_0^\infty
\left(e^{t\Delta}f(x)-f(x)\right)\frac{dt}{t^{1+s}},
\end{equation}
see  \cite{Davi,Davies1}. In Section \ref{sec.proof.th1} we use will yet another equivalent characterization of the spectral fractional Laplacian in terms of the so-called cylinder Caffarelli-Silvestre extension, as introduced by \cite{Cabre-Tan} in this context.

 Our aim of this paper is prove  the existence of possibly sign-changing, weak solutions to  Problem \eqref{pro1} for any $m_1,m_2>0$. Moreover we show that these solutions satisfy a smoothing effect estimate and possess a universal bound when $f=0$.

Some previous literature: this equation has been studied in the whole space $\ren$ as a model for porous medium flows with fractional nonlocal pressure in the case  $m_1=m_2=1$ by Caffarelli and the second author in \cite{CV1}. It is the most relevant case of the class of equations of the more general form
\begin{equation*}
\qquad \frac{\partial \rho}{\partial t}=\nabla   \cdot\left[\sigma_s(\rho)\nabla
\frac{\delta F(\rho)}{\delta \rho }\right]\,,
\end{equation*}
that arise in the description of the macroscopic evolution of particle systems with long range interactions, \cite{GL97, GLM2000}. Here, $\rho(x,t)\ge 0$ is the macroscopic  density, $F$ is a free energy functional, and the mobility function $\sigma_s(\rho)\ge 0$ may be degenerate, i.e., it may vanish for some values of $\rho$ (in our case \eqref{pro1} we have $\sigma_s(\rho)=|\rho|^{m_1}$ that vanishes at $\rho=0$).

The same equation as in \cite{CV1} appears in a one-dimensional model in dislocation theory that has also been studied by Biler et al. \cite{BKM10}. Later mathematical works include \cite{CV2, CSV, CV3}, where regularity and asymptotic  behaviour are established, paper \cite{BiImKa} that treats the case $m_1=1$, $m_2>\max\{\frac{1-2s}{1-s},\frac{2s-1}{N}\}$, and the works \cite{StTsVz.CRAS, StTsVz16, StTsVz17} that treat the cases where $m_1\ne 1$, and \cite{StTsVz15} that treats general exponents, see also \cite{DZ2017}.

In the limit case $m_1=0$ we obtain a different type of equation
 $$
 \partial_t u+(-\Delta)^{1-s} ( |u|^{m-1}u)=f,
 $$
 that has received many contributions, starting with \cite{pqrv1, pqrv2}. In all those works the forcing term $f=0$ is put to zero.  See  \cite{VazCime} for a general reference on recent work on nonlinear diffusion.

No works seem to have treated the same problem posed in a bounded domain when $m_1\ne 0$. As said above, we address this issue in the case where the fractional operator $(-\Delta)^{-s}$ is the inverse of the spectral fractional Laplacian operator. Attention is also paid to $f\ne 0$.

\medskip

\noindent {\bf \large Definition and main results}

We introduce next our main contributions.
In this paper, we  put $\gamma:=m_1+m_2$, this parameter will appear often.
This is the definition of weak solution that we are going to use

\begin{definition} Let $u_0\in L^1(\Omega)$ and $f\in L^1(0,T,(W_0^{1,\infty}(\Omega))^*)$. We say that $u$ is a weak solution  of problem \eqref{pro1} if
	
\noindent (i) $u\in L^{\max\{1,\gamma\}}(\Omega_T)$,
	
\noindent (ii) $\operatorname{div}(|u|^{m_1}\nabla (-\Delta)^{-s}|u|^{m_2-1}u)\in L^1(0,T,(W_0^{2,\infty}(\Omega))^*)$,  and
	\begin{align*}
	-\int_{0}^{T}\int_\Omega u\phi_tdxdt+\int_{0}^{T} \langle \ |u|^{m_1}\nabla (-\Delta)^{-s}(|u|^{m_2-1}u),\nabla \phi\rangle \,dt= \int_\Omega\phi(0)u_0\,dx+\int_{0}^{T}\langle f(t),\phi(t)\rangle \,dt
	\end{align*}
	for all $\phi\in C_c^2(\Omega\times[0,T))$.
\end{definition}
In general, we can not have $|u|^{m_1}\nabla (-\Delta)^{-s}|u|^{m_2-1}u\in L^1(\Omega_T)$, thus we can not replace (ii) in the definition by $|u|^{m_1}\nabla (-\Delta)^{-s}|u|^{m_2-1}u\in L^1(0,T,(W_0^{1,\infty}(\Omega))^*)$. See more on this in Lemma \ref{le-f}.


The following result contains the basic existence and main properties.

\begin{theorem}\label{mainthm} Let  $u_0\in L^\infty(\Omega)$ and $f\in L^\infty(\Omega_T)$. Then, there exists a weak solution $u$ of Problem \eqref{pro1} such that $u\in L^\infty(\Omega_T)$, $(-\Delta)^{\frac{1-s}{2}}(|u|^{\gamma-1}u)\in L^2(\Omega_T)$. Moreover, $u$ has the following properties:

\noindent {\rm \bf (I)} Basic $L^1$ estimate: for every $t>0$
\begin{equation}
||u^\pm||_{L^\infty(0,T,L^1(\Omega))}\leq ||u_0^\pm||_{L^1(\Omega)}+||f^\pm||_{L^1(\Omega_T)}\,.
\end{equation}
In particular, If $u_0,f\geq 0$, then $u\geq 0$ in $\Omega_T$.

\noindent {\rm \bf (Ia)}  We have the three-option estimate
	\begin{align}
&1_{s<1-\frac{N}{2}}||u||_{L^{\gamma+\frac{2(1-s)}{N},\infty}(\Omega_T)}+1_{s=1-\frac{N}{2}}||u||_{L^{\gamma+1-\frac{1}{r},\infty}(\Omega_T)}+1_{s>1-\frac{N}{2}}||u||_{L^{\gamma+1,\infty}(\Omega_T)}\\&\nonumber\leq C 1_{s<1-\frac{N}{2}}M^{\frac{N+2(1-s)}{\gamma N-2(1-s)}}+C1_{s=1-\frac{N}{2}}M^{\frac{2r}{r(\gamma+1)-1}}+C1_{s>1-\frac{N}{2}}M^{\frac{2}{\gamma+1}},
\end{align}
where $M=||u_0||_{L^1(\Omega)}+||f||_{L^1(\Omega_T)}$.

\noindent {\rm \bf (Ib)} Moreover,
\begin{align}
\int_{0}^T\int_\Omega \left|(-\Delta)^{\frac{1-s}{2}}\left(\frac{|u|^{\frac{\gamma}{2}+\theta-1}u}
{|u|^{2\theta}+1}\right)\right|^2dxdt\leq C(\theta)M~~ \forall \theta>0,
\end{align}

\noindent {\rm \bf (II)} for $p\in (1,\infty)$ and for all $t\in (0,T)$
\begin{align}\label{es62}
\int_\Omega |u(t)|^p& +\frac{4m_2p(p-1)}{(\gamma+p-1)^2}\int_{0}^{t}\,\int_{\Omega}
\left| (-\Delta)^{\frac{1-s}{2}}(|u|^{\frac{\gamma+p-1}{2}-1}u)\right|^2
\\& \leq \int_\Omega |u_0|^p+p\int_{0}^{t}\int_{\Omega} |f||u|^{p-1}. \nonumber
\end{align}

\noindent {\rm \bf (III)} $L^\infty$ bounds:
\begin{align}
||u||_{L^\infty(\Omega_T)}\leq ||u_0||_{L^\infty(\Omega)}+T||f||_{L^\infty(\Omega_T)}.
\end{align}

\noindent {\rm \bf(IV)} Smoothing effect: Assume $f=0$,
\begin{align}\label{sm.eff}
||u(t)||_{q}\leq  C ||u_0||_{L^{q_0}(\Omega)}^{\frac{N(\gamma-1)\frac{q_0}{q}+2q_0(1-s)}{N(\gamma-1)+
2q_0(1-s)}}t^{-\frac{(1-\frac{q_0}{q})N}{N(\gamma-1)+2q_0(1-s)}}~~\forall q\in [q_0,\infty]\,,
\end{align}
and
\begin{align}\label{sm.eff'}
\int_{\Omega}| (-\Delta)^{\frac{1-s}{2}}(|u(t)|^{\frac{\gamma+q-1}{2}})|^2dx\leq  C ||u_0||_{L^{q_0}(\Omega)}^{\frac{N(\gamma-1)q_0+2q_0q(1-s)}{N(\gamma-1)+
2q_0(1-s)}}t^{-\frac{(q-q_0)N}{N(\gamma-1)+2q_0(1-s)}-1}~\forall q\in [q_0,\infty)\cap(1,\infty).
\end{align}
provided $q_0\geq 1$ and $N(\gamma-1)+2q_0(1-s)>0$.
\end{theorem}
We would like to mention that estimates \noindent {\rm \bf (Ia)}  and \noindent {\rm \bf (Ib)}  for porous medium equations  were established in \cite{Veron-Nguyen}.

Whole proof of this result is given in Section \ref{sec.proof.th1}, where a number of other estimates are derived, see  Lemma 10. More general, unbounded data will be considered later as limits of this construction. The next result is called Universal Bound, a  very important property that is typical of Dirichlet problems in bounded domains and we can also prove in this generality. We have

\begin{theorem}\label{mainthm2} Let $\gamma>1$, $f=0$ and $u_0\in L^1(\Omega)$. Let $u$ be a solution of Problem \eqref{pro1} as constructed  in Theorem \ref{mainthm}. There exists $C=C(N,s,\gamma,\Omega)$ such that
	\begin{align}\label{univ.bound}
	||u(\cdot,t)||_{L^\infty(\Omega)}\leq C\, t^{-\frac{1}{\gamma-1}}~~\forall t>0.
	\end{align}
\end{theorem}

This is proved in Section \ref{sec.univbdd}. The point is that the estimate does not depend on the  norm of the data, so it will hold for any solution that is obtained as limit of the constructed solutions, a fact that will be used in the last section. Note that the estimate is not useful for $t\sim 0$, but is  very efficient for large times since we expect the positive solutions to have precisely that size. On the other hand, a universal bound does not hold for $\gamma\le 1$, see details in Section \ref{sec.univbdd}.

Our study is completed with two theorems on the existence of solutions to Problem \eqref{pro1} with bad data, which are contained in Section \ref{sec.bad}. Statements and full proofs are given there.

\medskip

\noindent {\sc Some notations.}  By $1_A$ we denote the characteristic function of the set $A$. \\We will use the distance to the boundary defined as
\begin{equation}\label{notatioin-delta}
d(x)=d_{\partial\Omega}(x) = \mbox{dist} (x, \partial\Omega) :=
\{\inf | x- y|: {y \in \partial\Omega} \}
\end{equation}
 for $x\in\Omega$. We put $\Omega_{\varepsilon}=\{x\in \Omega:d(x)<\varepsilon\}$.

We gather in Section \ref{sec.prelim} a list of  facts on the Heat Equation that we use in deriving properties of the semigroup generated by the spectral fractional Laplacian.

\section{Approximation of the fractional Laplacian $(-\Delta)^\alpha$}
\label{sec.apprx}

Let $\alpha\in (0,1), \varepsilon>0$ and $f\in L^\infty(\Omega)$. We define the operator,
\begin{align}
\mathcal{L}^\alpha_\varepsilon[f](x):=\int_{\varepsilon}^{\infty}(f(x)-e^{t\Delta}f(x))t^{-1-\alpha}dt.
\end{align}
Clearly, the following two properties are true: \\
{\bf 1.} Positivity
\begin{align}\label{es3}
	\int_{\Omega}\mathcal{L}^\alpha_\varepsilon[f](x)f(x)dx\geq 0,
	\end{align}
	since $||f||_{L^2(\Omega)}\geq \int_{\Omega}e^{t\Delta}f(x) f(x)dx$ for all $t>0$.

 \noindent {\bf 2.} We have
  \begin{align}\label{es4}
	||\mathcal{L}^\alpha_\varepsilon[f]||_{L^\infty(\Omega)}\leq \frac{2}{\alpha\varepsilon^\alpha}||f||_{L^\infty(\Omega)}.
	\end{align}

\

\begin{lemma}\label{lem1} Let $f\in L^\infty(\Omega)$. Then,
	\begin{description}
		\item[1.](C\'ordoba-C\'ordoba inequality) for any $C^2$-convex function $\Phi$ satisfying  $\Phi(0)=0$ and for $\varepsilon\in (0,1]$, there holds
		\begin{align}\label{kato1}
		\Phi'(f)\mathcal{L}^\alpha_\varepsilon[f](x)\geq \mathcal{L}^\alpha_\varepsilon[\Phi(f)](x)~~\forall~~x\in \Omega.
		\end{align}
		Moreover, if $(-\Delta)^\alpha f\in L^1(\Omega)$,
		\begin{align}\label{kato2}
		\Phi'(f)(-\Delta)^\alpha f(x)\geq (-\Delta)^\alpha\Phi(f)(x)~~\forall~~x\in \Omega.
		\end{align}
		\item[2.] for any $\delta\in (0,1-\alpha)$, there is $C_\delta>0$ such that
			\begin{align}\label{con-L^2}
	 ||(-\Delta)^{-1}\mathcal{L}^\alpha_\varepsilon[f]-(-\Delta)^{-1+\alpha}f||_{L^2(\Omega)}\leq C_{\delta}\varepsilon^{1-\alpha-\delta}||f||_{L^2(\Omega)},
		\end{align}
		for all  $0<\varepsilon<1$.

\item[3.]  for any $\delta\in (0,1-\alpha)$, there is $C_\delta>0$ such that
		\begin{align}\label{con-H^1}
		||(-\Delta)^{-1/2}\mathcal{L}^\alpha_\varepsilon[f]-(-\Delta)^{-1/2+\alpha}f||_{L^2(\Omega)}\leq C_{\delta}\varepsilon^{1-\alpha-\delta}||f||_{H^1_0(\Omega)},
		\end{align}
		for all  $0<\varepsilon<1$. In particular,
			\begin{align}\label{con-H^1'}
		\sup_{\varepsilon\in (0,1)}||(-\Delta)^{-1/2}\mathcal{L}^\alpha_\varepsilon[f]||_{L^2(\Omega)}\leq C||f||_{H^1_0(\Omega)}.
		\end{align}
	\end{description}
\end{lemma}
\begin{proof}1. Estimates \eqref{kato1} and \eqref{kato2} were proved in  \cite{Con-Igna}.

	2. We have
	\begin{align*}
	&I=\int_{\Omega}(-\Delta)^{-1}\mathcal{L}^\alpha_\varepsilon[f](x)\varphi(x)dx-\int_\Omega(-\Delta)^{-1+\alpha}f(x)\varphi(x) dx\\&= c_\alpha\int_{0}^{\infty}\int_{0}^{\varepsilon}\left[\langle e^{t\Delta}f,\varphi\rangle-\langle e^{(t+\rho)\Delta}f,\varphi\rangle\right]\rho^{-1-\alpha}d\rho dt
	\\&= c_\alpha\int_{0}^{\infty}\int_{0}^{\varepsilon}\int_{t}^{t+\rho}\langle -\frac{\partial }{\partial\tau}e^{\tau\Delta}f,\varphi\rangle d\tau\rho^{-1-\alpha}d\rho dt
	\\&= c_\alpha\int_{0}^{\infty}\int_{0}^{\varepsilon}\int_{t}^{t+\rho}\langle (-\Delta)e^{\tau\Delta}f,\varphi\rangle d\tau\rho^{-1-\alpha}d\rho dt.
	\end{align*}
Note that, in view of
\begin{align}\label{se-1}
-\Delta e^{\tau\Delta}f(x)=\sum_{j=1}^{\infty}\lambda_j e^{-\tau\lambda_j}\langle f,\varphi_j\rangle \varphi_j(x)~~\text{a.e }~~(x,\tau)\in \Omega\times (0,\infty),
\end{align} so, by H\"older's inequality and Plancherel's Theorem yields
\begin{align*}
|| (-\Delta) e^{\tau\Delta}f||_{L^2(\Omega)}\leq  C\left(e^{-\lambda_1\tau/2}1_{\tau>1/2}+\frac{1}{\tau}1_{\tau\leq 1/2}\right)||f||_{L^2(\Omega)}.
\end{align*}
Thus,
\begin{align*}
|I|&\leq C \int_{0}^{\infty}\int_{0}^{\varepsilon}\int_{t}^{t+\rho}\left(e^{-\lambda_1\tau/2}1_{\tau>1/2}+\frac{1}{\tau}1_{\tau\leq 1/2}\right) d\tau\rho^{-1-\alpha}d\rho dt ||f||_{L^2(\Omega)} ||\varphi||_{L^2(\Omega)}\\& :=C L(\varepsilon) ||f||_{L^2(\Omega)} ||\varphi||_{L^2(\Omega)}.
\end{align*}
It is enough to check that
\begin{align}
L(\varepsilon) \leq C_\delta\varepsilon^{1-\alpha-\delta}.
\end{align}
Indeed,
\begin{align*}
L(\varepsilon)&\leq  \int_{1/2}^{\infty}\int_{0}^{\varepsilon}\int_{t}^{t+\rho}e^{-\lambda_1\tau/2}d\tau\rho^{-1-\alpha}d\rho dt +\int_{0}^{1/2}\int_{0}^{\varepsilon}\int_{t}^{t+\rho}\frac{1}{\tau} d\tau\rho^{-1-\alpha}d\rho dt\\&\leq \int_{1/2}^{\infty}\int_{0}^{\varepsilon}e^{-\lambda_1t/2}\rho^{-\alpha}d\rho dt +\int_{0}^{1/2}\int_{0}^{\varepsilon}\log\left(1+\frac{\rho}{t}\right) \rho^{-1-\alpha}d\rho dt
\\&\leq C\varepsilon^{1-\alpha} +C_\delta\varepsilon^{1-\alpha-\delta}\\&\leq C_\delta\varepsilon^{1-\alpha-\delta}.
\end{align*}

\noindent 3. 	As above, we have for any $\varphi\in L^2(\Omega)$, \begin{align*}
&II:=\int_{\Omega}(-\Delta)^{-1/2}\mathcal{L}^\alpha_\varepsilon[f]\varphi dx-\int_{\Omega}\varphi(-\Delta)^{-1/2+\alpha}f dx\\&= c_\alpha\int_{0}^{\infty}\int_{0}^{\varepsilon}\int_{t}^{t+\rho}\langle (-\Delta)e^{\tau\Delta}f,\varphi\rangle t^{-1/2}\rho^{-1-\alpha}d\tau d\rho dt.
\end{align*}
We deduce from \eqref{se-1} that
\begin{align*}
||(-\Delta) e^{\tau\Delta}f||_{L^2(\Omega)}\leq C\left(e^{-\lambda_1\tau/2}1_{\tau>1/2}+\frac{1}{\sqrt{\tau}}1_{\tau\leq 1/2}\right)||f||_{H^1_0(\Omega)}.
\end{align*}
Thus,
\begin{align*}
|II|\leq C \int_{0}^{\infty}\int_{0}^{\varepsilon}\int_{t}^{t+\rho}\left(e^{-\lambda_1\tau/2}1_{\tau>1/2}+\frac{1}{\sqrt{\tau}}1_{\tau\leq 1/2}\right) t^{-1/2}\rho^{-1-\alpha}d\tau d\rho dt ||f||_{H^1_0(\Omega)} ||\varphi||_{L^2(\Omega)}.
\end{align*}
Since,
\begin{align*}
&\int_{0}^{\infty}\int_{0}^{\varepsilon}\int_{t}^{t+\rho}\left(e^{-\lambda_1\tau/2}1_{\tau>1/2}+\frac{1}{\sqrt{\tau}}1_{\tau\leq 1/2}\right) t^{-1/2}\rho^{-1-\alpha}d\tau d\rho dt \\&\leq C\varepsilon^{1-\alpha}+C\int_{0}^{1/2}\int_{0}^{\varepsilon} \left(\sqrt{t+\rho}-\sqrt{t}\right) t^{-1/2}\rho^{-1-\alpha}d\rho dt
\\&\leq C_\delta\varepsilon^{1-\alpha-\delta},
\end{align*}
thus, we get \eqref{con-H^1}.
 The proof is complete.
\end{proof}

\begin{remark}\label{re2} In the proof of \eqref{con-L^2} we also get for any $0<\alpha<\alpha_0<1$,
	\begin{align}\label{es65}
	||\mathcal{L}_\varepsilon^\alpha[f]-(-\Delta)^\alpha f||_{L^2(\Omega)}\leq C \varepsilon^{\alpha_0-\alpha}||(-\Delta)^{\alpha_0} f||_{L^2(\Omega)}.
	\end{align}
\end{remark}
\begin{remark}\label{re1} From Lemma \ref{lem1}, we have for all $f\in L^\infty(\Omega)$,
	\begin{align}\label{mono1}
		&	\int_\Omega H(f)\mathcal{L}_\varepsilon^\alpha [G(f)]dx\geq 0,
	\end{align}
	where $G,H\in C^2(\Omega)$ are  strictly increasing functions.
	Moreover, for all $f\in L^\infty(\Omega)$ and $(-\Delta)^{\alpha}(f)\in L^1(\Omega)$
	\begin{align} \label{mono2}	\int_\Omega H(f)(-\Delta)^\alpha G(f)dx \geq 0.
	\end{align}
\end{remark}

\begin{remark} Using \eqref{kato2} yields \begin{align*}
	\int_\Omega |u|^2u(-\Delta)^\alpha u dx\geq \frac{1}{2}\int_\Omega |u|^2(-\Delta)^\alpha |u|^2 dx\geq \frac{1}{2}\int_{\Omega}|(-\Delta)^{\alpha/2} |u|^2|^2 dx.
	\end{align*}
	for $u\in C_c^\infty(\Omega)$.
Unfortunately, we can not have
 \begin{align*}
C \int_{\Omega}|(-\Delta)^{\alpha/2} |u|^2|^2 dx\geq \int_{\Omega}|(-\Delta)^{\alpha/2} (|u|u)|^2 dx.
 \end{align*}
 By this way, we can not find
  \begin{align*}
 \int_\Omega |u|^2u(-\Delta)^\alpha u dx\geq C\int_{\Omega}|(-\Delta)^{\alpha/2} (|u|u)|^2 dx.
 \end{align*}
Therefore, we next prove this inequality by another way.  It  is a  version of the so-called Stroock-Varadhan inequality, we refer to \cite{Stoo} and \cite{Liskevich} where this kind of inequality is proved for {\sl general sub-markovian operators}.
\end{remark}

\begin{lemma}[Stroock-Varopoulos inequality for $(-\Delta)^\alpha$] \label{le-Stook-Varopoulos}Let $\psi:\mathbb{R}\to\mathbb{R}$ such that $\psi\in C^2(\mathbb{R})$ and $\psi'\geq 0$. Then,
	\begin{align}\label{es73}
	\int_\Omega \psi(u)(-\Delta)^\alpha udx\geq \int_\Omega|(-\Delta)^{\frac{\alpha}{2}}\Psi(u)|^2dx,
	\end{align}
	where $\psi'=(\Psi')^2$.
\end{lemma}
\begin{proof} To prove this, we will use the Stinga-Torrea extension problem in \cite{Stinga}, which is in turn a generalization of the Caffarelli-Silvestre estension problem in \cite{Caffa-Sil}. For the equivalence of this problem with the original problem with the spectral Laplacian see for instance \cite{Cabre-Tan, pqrv1, pqrv2}.  Let $U,V$ be  unique solutions  of the extended problems
	\begin{equation*}
	\left\{
	\begin{array}
	{ll}%
	\operatorname{div}_{x,y}(y^{1-2\alpha}\nabla_{x,y} U)=0~~&\text{in }\Omega\times(0,\infty),\\
	U=0 &\text{on}
	~\partial\Omega\times (0,\infty),
	\\
	U(x,0)=u(x) &\text{in}
	~~\Omega,
	\\
	\end{array}
	\right.
	\end{equation*}
	\begin{equation*}
	\left\{
	\begin{array}
	{ll}%
	\operatorname{div}_{x,y}(y^{1-2\alpha}\nabla_{x,y} V)=0~~&\text{in }\Omega\times(0,\infty),\\
	V=0 &\text{on}
	~~\partial\Omega\times (0,\infty),
	\\
	V(x,0)=\Psi(u(x)) &\text{in}
	~~\Omega,
	\\
	\end{array}
	\right.
	\end{equation*}
	resp. . By the extension theorem (see \cite{Stinga}), we have
	\begin{align}\label{es71}
	\int_{\Omega}\int_0^\infty y^{1-2\alpha}\nabla_{x,y} U\nabla_{x,y} \varphi dydx=c_\alpha\int_{\Omega}(-\Delta)^{\alpha}(v) \varphi(0)dx,\end{align}
	and
	\begin{align}\label{es71*}
	\int_{\Omega}\int_0^\infty y^{1-2\alpha}\nabla_{x,y} V\nabla_{x,y} \varphi dydx=c_\alpha\int_{\Omega}(-\Delta)^{\alpha}(\Psi(u)) \varphi(0)dx,\end{align}
	for any $\varphi\in H^1_0(\Omega\times (0,\infty),d\mu)$ with $d\mu= y^{1-2\alpha}dydx$.\\
	Applying \eqref{es71} to $\varphi=\psi(U)$ and \eqref{es71*} to $\varphi=V$ and using $\psi'=(\Psi')^2$, we get
		\begin{align*}
	\int_{\Omega}\int_0^\infty y^{1-2\alpha}|\nabla_{x,y} \Psi(U)|^2 dydx=c_\alpha\int_{\Omega}(-\Delta)^{\alpha}(v) \psi(u)dx,\end{align*}
	and
		\begin{align*}
	\int_{\Omega}\int_0^\infty y^{1-2\alpha}|\nabla_{x,y} V|^2 dydx=c_\alpha\int_{\Omega}|(-\Delta)^{\alpha/2}(\Psi(u))|^2dx.\end{align*}
	Thus, it is enough to show that
	\begin{align}\label{es72}
		\int_{\Omega}\int_0^\infty y^{1-2\alpha}|\nabla_{x,y} \Psi(U)|^2 dydx\geq 	\int_{\Omega}\int_0^\infty y^{1-2\alpha}|\nabla_{x,y} V|^2 dydx.
	\end{align}
	Indeed, since $\operatorname{div}(y^{1-2\alpha}\nabla_{x,y}(\Psi(U)-V))=\operatorname{div}(y^{1-2\alpha}\nabla_{x,y}\Psi(U))$
	\begin{align*}
	\int_{\Omega}\int_0^\infty y^{1-2\alpha}|\nabla_{x,y} (\Psi(U)-V)|^2 dydx=\int_{\Omega}\int_0^\infty y^{1-2\alpha}\nabla_{x,y}\Psi(U) \nabla_{x,y} (\Psi(U)-V) dydx,
	\end{align*}
	it follows
	\begin{align*}
		\int_{\Omega}\int_0^\infty y^{1-2\alpha}|\nabla_{x,y} V|^2 dydx=	\int_{\Omega}\int_0^\infty y^{1-2\alpha}\nabla_{x,y}\Psi(U) \nabla_{x,y} V dydx.
	\end{align*}
	Using H\"older's inequality we find \eqref{es72}. The proof is complete.
\end{proof}
\begin{corollary}\label{stook-var-in-coro} Let $q_1,q_2>0$. Then,
	\begin{align}\label{es74}
	\int_\Omega |u|^{q_1-1}u(-\Delta)^\alpha(|u|^{q_2-1}u) dx\geq \frac{4q_1q_2}{(q_1+q_2)^2}\int_\Omega|(-\Delta)^{\frac{\alpha}{2}}(|u|^{\frac{q_1+q_2}{2}-1}u)|^2dx.
	\end{align}
\end{corollary}
\begin{proof} Set $v=|u|^{q_2-1}u$ and $\psi(v)=|v|^{\frac{q_1}{q_2}-1}v$ and $\Psi(v)=\left[\frac{4q_1q_2}{(q_1+q_2)^2}\right]^{1/2}|v|^{\frac{q_1}{2q_2}+\frac{1}{2}}$. We have,$
	\psi(v)=|u|^{q_1-1}u,\psi'=[\Psi']^2.$
	Thus, it follows \eqref{es74} from Lemma \ref{le-Stook-Varopoulos}. The proof is complete.
\end{proof}

\section{A regularized problem}\label{sec.reg}

In this section, we will prove existence of solutions to the following regularized problem:
\begin{equation}\label{pro4}
\left\{
\begin{array}
[c]{l}%
\partial_t u-\delta\Delta u-\operatorname{div}(|u|^{m_1}\nabla (-\Delta)^{-s} |u|^{m_2-1}u)=f~~\text{in }\Omega_T,\\
u=0~~~~~~~~~~~~~~~~~~~~~~~~~~~~~~~~~~~~~~~~~~~~~~~~~~~~~~\text{on}
~\partial\Omega\times (0,T),
\\
u(0)=u_0~~~~~~~~~~~~~~~~~~~~~~~~~~~~~~~~~~~~~~~~~~~~~~~~~\text{in}
~~\Omega,
\\
\end{array}
\right.
\end{equation}
with $\delta\in (0,1)$.
\begin{theorem} \label{step4}  Let $u_0\in L^\infty(\Omega)$ and $f\in L^\infty(\Omega_T)$.
	Then, there exists a weak solution
	$$u\in L^\infty(\Omega_T)\cap C(0,T;L^2(\Omega))\cap L^2(0,T;H^1_0(\Omega))$$
	of problem \eqref{pro4}.
\end{theorem}
In this section, we set
\begin{align}
X:=L^\infty(\Omega_T)\cap C(0,T;L^2(\Omega))\cap L^2(0,T;H^1_0(\Omega)).
\end{align}
Here and in what follows, we use the following definition:
\begin{definition} Let $u_0\in L^\infty(\Omega)$ and $f\in L^\infty(\Omega_T)$. We say that
	$u\in X$ is a weak solution  of
	\begin{equation}\label{pro2*}
	\left\{
	\begin{array}
	[c]{l}%
	\partial_tu-\delta\Delta u+\mathcal{F}(u)=f~~\text{in }\Omega_T,\\
	u=0~~~~~~~~~~~~~~~~~~~~~~~\text{on}
	~~\partial\Omega\times (0,T),
	\\
	u(0)=u_0~~~~~~~~~~~~~~~~~~\text{in}
	~~\Omega,
	\\
	\end{array}
	\right.
	\end{equation}
	if $\mathcal{F}(u)\in L^1(0,T,(W_0^{2,\infty}(\Omega))^*)$ and
	\begin{align*}
	\int_{0}^{T}\int_\Omega u(-\varphi_t-\delta\Delta \varphi)dxdt-\int_{0}^{T} \langle \mathcal{F}(u),\varphi\rangle dt=\int_\Omega\varphi(0)u_0 dx+\int_{0}^{T}\int_\Omega \varphi f dxdt
	\end{align*}
	for all $\varphi\in C_c^1([0,T), (W_0^{2,\infty}(\Omega))^*)$.
\end{definition}
In order to construct the weak solution of problem  \eqref{pro4}, we first consider the following  problem
\begin{equation}\label{pro2}
\left\{
\begin{array}
[c]{l}%
\partial_tu-\delta\Delta u-\operatorname{div}(H_{\kappa_2}(|u|)\nabla (-\Delta)^{-1}\mathcal{L}_\varepsilon^{1-s} [G_{\kappa_2}(u)])+\varpi\mathcal{L}_\varepsilon^{s_0} J_{\kappa_1}(u)=f~~\text{in }\Omega_T,\\
u=0~~~~~~~~~~~~~~~~~~~~~~~~~~~~~~~~~~~~~~~~~~~~~~~~~~~~~~~~~~~~~~~~~~~~~~~~~~~~~~~~~~~\text{on}
~\partial\Omega\times (0,T),
\\
u(0)=u_0~~~~~~~~~~~~~~~~~~~~~~~~~~~~~~~~~~~~~~~~~~~~~~~~~~~~~~~~~~~~~~~~~~~~~~~~~~~~~~\text{in}
~~\Omega,
\\
\end{array}
\right.
\end{equation}
where $s_0=\frac{(1-2s)^++1}{2}\in (0,1)$, $\varpi,\kappa_1,\kappa_2\in (0,1)$ and
\begin{align*}
J_{\kappa_1}(u)=\frac{|u|^{m_0+1}u}{u^2+\kappa_1},~ H_{\kappa_2}(|u|)=\frac{|u|^{m_1+2}}{u^2+\kappa_2},~G_{\kappa_2}(u)=\frac{|u|^{m_2+1}u}{u^2+\kappa_2}~\text{with}~m_0=\frac{1}{8}\min\{m_1,m_2\}.
\end{align*}

\begin{proposition} \label{step1}Let $f\in L^\infty(\Omega_T),u_0\in L^\infty(\Omega)$. Then, problem \eqref{pro2} admits a weak solution $u\in C(0,T,L^\infty(\Omega))\cap L^2(0,T, H^1_0(\Omega))$.
\end{proposition}
\begin{proof} Let $T_0\in (0,1)$. We consider
	\begin{align*}
	\mathcal{T}:v\mapsto e^{t\Delta}u_0+\int_{0}^{t}e^{\delta(t-\tau)\Delta}  \Theta(v,f)d\tau,
	\end{align*}
	 for $v\in L^\infty(\Omega_{T_0})$, where
	 \begin{align*}
	 \Theta(v,f)=\operatorname{div}(H_{\kappa_2}(|v|)\nabla (-\Delta)^{-1}\mathcal{L}_\varepsilon^{1-s} [G_{\kappa_2}(v)])-\varpi\mathcal{L}_\varepsilon^{s_0} J_{\kappa_1}(v)+f.
	 \end{align*}
	 Using \eqref{es4} and  \eqref{para3} with $u_0=0$ and $g=H_{\kappa_2}(|v|)\nabla (-\Delta)^{-1}\mathcal{L}_\varepsilon^{1-s} [G_{\kappa_2}(v)]$  yields
	 \begin{align*}
	 &\left|e^{\delta(t-\tau)\Delta}\operatorname{div}(H_{\kappa_2}(|v|)\nabla (-\Delta)^{-1}\mathcal{L}_\varepsilon^{1-s} [G_{\kappa_2}(v)])\right|\leq \frac{C}{\sqrt{t-\tau}} ||v||_{L^\infty(\Omega)}^{\gamma+4},\\&
	 \left|e^{\delta(t-\tau)\Delta} \mathcal{L}_\varepsilon^{s_0} J_{\kappa_1}(v)\right|\leq C||v||_{L^\infty(\Omega)}^{m_0+2},
	 \end{align*}
	for any $0<\tau<t$, where $C=C(\varepsilon,\kappa_1,\kappa_2,s_0,s,N,\Omega)$.
Thus, the operator $\mathcal{T}$ is well-defined and map from $L^\infty(\Omega_{T_0})$ into itself. Moreover, since $ \Theta(v,f)\in L^\infty(0,T_0,(W^{1,1}_0(\Omega))^*)+L^\infty(\Omega_{T_0})$, so by standard properties, we have
\begin{align}\label{es64}
\mathcal{T} (L^\infty(\Omega_{T_0}))\subset C(0,T_0,L^\infty(\Omega))\cap L^2(0,T_0, H^1_0(\Omega)).
\end{align} Next,
we show that $\mathcal{T}$ has a fixed point by the Banach contraction principle provided that $T_0=T_0(||u_0||_{L^\infty(\Omega)},\varepsilon,H,G,\delta)$. To do that, we have the following claim:

\textbf{Claim:} for any $K>0$,  for all $u,v\in \overline{B}(0,K)\subset L^\infty(\Omega_{T_0})$ there holds
\begin{align}\label{contrac}
||\mathcal{T}(u)-\mathcal{T}(v)||_{L^\infty(\Omega_{T_0})}\leq C_1(K)\sqrt{T_0}||u-v||_{L^\infty(\Omega_{T_0})},
\end{align}
where $C_1(K)$ is a constant which also depend on $s,N,\kappa_1,\kappa_2,\varepsilon,\Omega,K.$ Indeed, set \begin{align*}
E:=H_{\kappa_1}(|u|)\nabla (-\Delta)^{-1}\mathcal{L}_\varepsilon^{1-s} [G_{\kappa_1}(u)]-H_{\kappa_1}(|v|)\nabla (-\Delta)^{-1}\mathcal{L}_\varepsilon^{1-s} [G_{\kappa_1}(v)]
\end{align*}
We have,
\begin{align*}
&||E||_{L^\infty(\Omega)}\leq C(K)||u-v||_{L^\infty(\Omega)}||\nabla (-\Delta)^{-1}\mathcal{L}_\varepsilon^{1-s} [G_{\kappa_2}(u)]||_{L^\infty(\Omega)}\\&~+C(K)||\nabla (-\Delta)^{-1}\mathcal{L}_\varepsilon^{1-s} [G_{\kappa_2}(u)-G_{\kappa_2}(v)]||_{L^\infty(\Omega)}
\\&\overset{\eqref{es2}}\leq C(K)||u-v||_{L^\infty(\Omega)}||\mathcal{L}_\varepsilon^{1-s} [G_{\kappa_2}(u)]||_{L^\infty(\Omega)}+C(K)||\mathcal{L}_\varepsilon^{1-s} [G_{\kappa_2}(u)-G_{\kappa_2}(v)]||_{L^\infty(\Omega)}
\\&\overset{\eqref{es4}}\leq C(K)||u-v||_{L^\infty(\Omega)}+C(K)||G_{\kappa_2}(u)-G_{\kappa_2}(v)||_{L^\infty(\Omega)}
\\&\leq C(K)||u-v||_{L^\infty(\Omega)},
\end{align*}
where $C(K)$ is a constant which also depend on $s,N,\kappa_1,\kappa_2,\varepsilon,\Omega,K.$\\
Using \eqref{para3+}  in Lemma \ref{le12} with  $g=E$,  we get  for $t\in (0,T_0)$,
\begin{align*}
&|\mathcal{T}(u)(t)-\mathcal{T}(v)(t)|\\&\leq \left|\int_{0}^{t}e^{\delta(t-\tau)\Delta}\operatorname{div}(E(\tau)) d\tau\right|+\left|\int_{0}^{t}e^{\delta(t-\tau)\Delta} \mathcal{L}_\varepsilon^{s_0} \left(J_{\kappa_1}(v)- J_{\kappa_1}(u)\right)d\tau\right|\\ & \leq C(K)\int_{0}^{t}\frac{1}{\sqrt{t-\tau}} d\tau ||u-v||_{L^\infty(\Omega_{T_0})}+C\int_{0}^{t}d\tau ||J_{\kappa_1}(v)- J_{\kappa_1}(u)||_{L^\infty(\Omega_{T_0})}
\\ & \leq C(K)\sqrt{t} ||u-v||_{L^\infty(\Omega_{T_0})}.
\end{align*}
It follows \eqref{contrac}. Thus, we get for $u\in \overline{B}(0,K)\subset C(0,T_0;L^\infty(\Omega)) $
\begin{align*}
||\mathcal{T}(u)||_{L^\infty(\Omega_{T_0})}&\leq ||\mathcal{T}(0)||_{L^\infty(\Omega_{T_0})} + C_1(K)\sqrt{T_0}||u||_{L^\infty(\Omega_{T_0})}\\&\leq C||f||_{L^\infty(\Omega_T)}+2||u_0||_{L^\infty(\Omega)}+KC_1(K)\sqrt{T_0}.
\end{align*}
Now,  choosing $K=2(C||f||_{L^\infty(\Omega_T)}+2||u_0||_{L^\infty(\Omega)})$ and $T_0=\frac{1}{4(C_1(K))^2}$ yields
\begin{align}
||\mathcal{T}(u)||_{L^\infty(\Omega_{T_0})}\leq K.
\end{align}
This means, $\mathcal{T}$ maps $\overline{B}(0,K)$ into itself and is a contraction. Hence, $\mathcal{T}$ has a fixed point in $L^\infty(\Omega_{T_0})$ for some $T_0>0$.\medskip\\ On the other hand, if $\mathcal{T}(u)=u$ in $L^\infty(\Omega_{T_1})$ then for all $q\geq3$ and $t\in (0,T_1)$
\begin{align*}
\int_\Omega |u(t)|^q&+\delta(q-1)\int_{0}^{t}\int_\Omega|u|^{q-1}|\nabla u|^2+(q-2)\int_{0}^{t}\int_\Omega |u|^{q-2} H_{\kappa_2}(|u|)\nabla (-\Delta)^{-1}\mathcal{L}_\varepsilon^{1-s}[ G_{\kappa_2}(u)]\nabla u \\&~~~+\int_{0}^{t}\int_\Omega (-\Delta)^{s_0} J_{\kappa_1}(u) |u|^{q-2}u =\int_{0}^{t}\int_{\Omega}f(t)|u(t)|^{q-1}u(t)+\int_\Omega|u_0|^q.
\end{align*}
Since\begin{align*}
&(q-2)\int_\Omega |u|^{q-2} H_{\kappa_2}(|u|)\nabla (-\Delta)^{-1}\mathcal{L}_\varepsilon^{1-s} [G_{\kappa_2}(u)]\nabla udx\\&~~~~~=(q-2)\int_{\Omega} \nabla (-\Delta)^{-1}\mathcal{L}_\varepsilon^{1-s} [G_{\kappa_2}(u)]\nabla \tilde{H}_{\kappa_2}(u)dx
\\&~~~~~=(q-2)\int_{\Omega} \mathcal{L}_\varepsilon^{1-s} [G_{\kappa_2}(u)]\tilde{H}_{\kappa_2}(u)dx\overset{\eqref{mono1}~\text{ in Remark}~\ref{re1}}\geq 0,
\end{align*}
with $\tilde{H}_{\kappa_2}(a)=\int_{0}^{a} |y|^{q-2} H_{\kappa_2}(|y|)dy$ and $
\int_\Omega \mathcal{L}_\varepsilon^{s_0} J_{\kappa_1}(u) |u|^{q-2}u\geq 0$,  thus, for $t\in (0,T_1)$
\begin{align*}
\sup_{\tau\in [0,t]}\int_\Omega |u(\tau)|^q\leq \left(\int_{0}^{t}\int_{\Omega}|f|^q\right)^{1/q}\left(\int_{0}^{t}\int_{\Omega}|u|^q\right)^{(q-1)/q}+\int_\Omega|u_0|^qdx.
\end{align*}
Using H\"older's inequality we obtain
\begin{align}
\sup_{\tau\in [0,T_1]}||u(\tau)||_{L^q(\Omega)}\leq CT_1||f||_{L^q(\Omega_{T_1})}+2||u_0||_{L^q(\Omega)},
\end{align}
where $C$ does not depend on $q$.
Letting $q\to \infty$, we deduce,
\begin{align}
\sup_{\tau\in [0,T_1]}||u(\tau)||_{L^\infty(\Omega)}\leq CT_1||f||_{L^\infty(\Omega_{T_1})}+2||u_0||_{L^\infty(\Omega)}.
\end{align}
In particular, the norm $||u(T_1)||_{L^\infty(\Omega)}$ cannot explode for $T_1<T$.
Thus, there exists $u\in L^\infty(\Omega_T)$ such that $\mathcal{T}(u)=u$. By \eqref{es61}, $u\in C(0,T,L^\infty(\Omega))\cap L^2(0,T;H^1_0(\Omega))$. Hence,
$u$ is a weak solution of \eqref{pro2}.
 The proof is complete.
\end{proof}
\begin{remark} By standard regularity, we can see that the solution of $u$ in Proposition \eqref{step1} belongs to  $W^{1,r}(\tau,T;W^{2,r}(\Omega))$ for all $r<\infty$ and $\tau\in (0,T)$. Moreover, if $u_0,f$ are smooth functions, then $u$ is too.
\end{remark}
The following is a variant of Simon's compactness Lemma for Space $L^1(0,T;X)$ which will be used several times in this paper.
\begin{lemma}\label{Simonlemma} Let $(v_n)\subset L^1(\Omega_T)$ be such that
	\begin{align}\label{es75}
	||v_{n}||_{L^q(\Omega_T)}+|||v_n|^{\alpha_1-1}v_n||_{L^1(0,T;W^{\alpha_2,1}(\Omega))}+||\frac{\partial}{\partial t}v_n||_{L^1(0,T;(W_0^{2,\infty}(\Omega))^*)}\leq C~~\forall~~n.
	\end{align}
	with $\alpha_1>0,q>1,\alpha_2\in (0,1)$. There exists a subsequence of $\{v_n\}$ converging to $v$ in $L^1(\Omega_T)$. 	
\end{lemma}
\begin{proof} If $\alpha_1\geq 1$, we have
	\begin{align*}
	||v_n||_{W^{\frac{\alpha_2}{\alpha_1},\alpha_1}(\Omega)}\leq C |||v_n|^{\alpha_1-1}v_0||_{W^{\alpha_2,1}(\Omega)}^{\frac{1}{\alpha_1}}\leq C,
	\end{align*}
	for all $n\in \mathbb{N}$. Thus,  by Simon's compactness Lemma, see \cite[Theorem 1 and Lemma 4]{Simon}, we find the conclusion for case $\alpha_1\geq 1$.\\
	We now consider case $\alpha_1\in (0,1)$. Since $L^q(\Omega)\subset (W_0^{2,\infty}(\Omega))^*$ is compact and
	\begin{align*}
	||v_{n}||_{L^q(\Omega_T)}+||\frac{\partial}{\partial t}v_n||_{L^1(0,T;(W_0^{2,\infty}(\Omega))^*)}\leq C~~\forall~~n,
	\end{align*}
	by Simon's compactness Lemma, see \cite[Theorem 1 and Lemma 4]{Simon}, there exists a subsequence $\{v_{n_k}\}$ of $\{v_{n}\}$  converging to $v$ in $L^1(0,T;(W_0^{2,\infty}(\Omega))^*)$. \\
	By a standard compact argument, see \cite[Lemma 8]{Simon}, for any $\eta>0$, there is a constant $C_\eta$ such that
	\begin{align*}
	||w-v||_{L^1(\Omega)}^{\alpha_1}&\leq \eta \left(|||w|^{\alpha_1-1}w-|v|^{\alpha_1-1}v||_{W^{\alpha_2,1}(\Omega)}+||w-v||_{L^q(\Omega)}^{\alpha_1}\right)\\&+C_\eta||w-v||_{(W_0^{2,\infty}(\Omega))^*}^{\alpha_1},
	\end{align*}
	for all $w\in L^q(\Omega), |w|^{\alpha_1-1}w\in W^{\alpha_2,1}(\Omega) $. This implies
	\begin{align*}
&\limsup_{k\to\infty}	||v_{n_k}-v||_{L^{\alpha_1}(0,T;L^1(\Omega))}^{\alpha_1}\\&\leq \eta \left(\limsup_{k\to\infty}|||v_{n_k}|^{\alpha_1-1}v_{n_k}-|v|^{\alpha_1-1}v||_{L^1(0,T;W^{\alpha_2,1}(\Omega))}+\limsup_{k\to\infty}||v_{n_k}-v||_{L^{\alpha_1}(0,T;L^q(\Omega))}^{\alpha_1}\right)\\&+C_\eta\limsup_{k\to\infty}||v_{n_k}-v||_{L^{\alpha_1}(0,T;(W_0^{2,\infty}(\Omega))^*)}\\&\leq C\eta +CC_\eta \limsup_{k\to\infty}||v_{n_k}-v||_{L^{1}(0,T;(W_0^{2,\infty}(\Omega))^*)}=C\eta.
	\end{align*}
	Letting $\eta\to 0$, we $v_{n_k}-v\to 0$ in $L^{\alpha_1}(0,T;L^1(\Omega))$. Finally, using an interpolation inequality we get  $v_{n_k}-v\to 0$ in $L^{1}(\Omega_T)$. The proof is complete.
\end{proof}
\begin{remark}\label{Simon-re} If $q=1$, we can show that there exists a subsequence of $\{v_n\}$ converging to $v$ in $L^{\theta}(\Omega_T)$ for all $\theta\in (0,1)$.
\end{remark}
\begin{proposition} \label{step2}Let $u_{\varepsilon}$ be a solution of problem \eqref{pro2} obtained in Proposition \ref{step1}. Then, there exists a subsequence of $\{u_{\varepsilon}\}$ converging to  a solution $u\in X$ of problem
	\begin{equation}\label{pro3}
	\left\{
	\begin{array}
	[c]{l}%
	\partial_tu-\delta\Delta u-\operatorname{div}(H_{\kappa_2}(|u|)\nabla (-\Delta)^{-s} G_{\kappa_2}(u))+\varpi(-\Delta)^{s_0} J_{\kappa_1}(u)=f~~\text{in }\Omega_T,\\
	u=0~~~~~~~~~~~~~~~~~~~~~~~~~~~~~~~~~~~~~~~~~~~~~~~~~~~~~~~~~~~~~~~~~~~~~~~~~~~\text{on}
	~\partial\Omega\times (0,T),
	\\
	u(0)=u_0~~~~~~~~~~~~~~~~~~~~~~~~~~~~~~~~~~~~~~~~~~~~~~~~~~~~~~~~~~~~~~~~~~~~~~\text{in}
	~\Omega,
	\\
	\end{array}
	\right.
	\end{equation}
	as $\varepsilon\to0$.
\end{proposition}
\begin{proof} Choosing $u_\varepsilon$ as test function in \eqref{pro2} we get
	\begin{align*}
|| u_\varepsilon||_{L^2(0,T;H^1_0(\Omega))}+||u_\varepsilon||_{L^\infty(\Omega_T)}\leq C~~\forall~~\varepsilon>0.
	\end{align*}
	By \eqref{con-H^1'} in Lemma \ref{lem1}, we have
	\begin{align*}
&	||\operatorname{div}(H_{\kappa_2}(|u_\varepsilon|)\nabla (-\Delta)^{-1}\mathcal{L}_\varepsilon^{1-s} [G_{\kappa_2}(u_\varepsilon)])||_{L^2(0,T;H^{-1}(\Omega))}\\&~~~~~~=	||H_{\kappa_2}(|u_\varepsilon|)\nabla (-\Delta)^{-1}\mathcal{L}_\varepsilon^{1-s} [G_{\kappa_2}(u_\varepsilon)]||_{L^2(\Omega_T)}
\\&~~~~~~\leq C	|| (-\Delta)^{-1/2}\mathcal{L}_\varepsilon^{1-s} [G_{\kappa_2}(u_\varepsilon)]||_{L^2(\Omega_T)}
\\&~~~~~~\leq C	|| G_{\kappa_2}(u_\varepsilon)||_{L^2(0,T;H^1(\Omega))}
\\&~~~~~~\leq C	.
	\end{align*}
	By \eqref{es65} in Remark \ref{re2}, for $s_1=\frac{s_0+1}{2}\in (s_0,1)$, we have
	\begin{align*}
	||\mathcal{L}_\varepsilon^{s_0} J_{\kappa_1}(u)||_{L^2(0,T;(H^1_0(\Omega)\cap H^{2s_1}(\Omega))^*)}\leq C.
	\end{align*}
	Thus,
	\begin{align*}
	||\partial_tu_\varepsilon-\delta\Delta u_\varepsilon||_{L^2(0,T;(H^1_0(\Omega)\cap H^{2s_1}(\Omega))^*)}+|| u_\varepsilon||_{L^2(0,T;H^1_0(\Omega))}+||u_\varepsilon||_{L^\infty(\Omega_T)}\leq C~\forall~\varepsilon\in (0,1).
	\end{align*}
	By Lemma \ref{Simonlemma}, there exists a subsequence of $\{u_\varepsilon\}$ converging to $u$ in $L^1(\Omega_T)$ as $\varepsilon\to0$. Moreover, we also have $u\in X$ and
$
	\lim\limits_{\varepsilon\to 0}	\mathcal{L}_\varepsilon^{s_0} [J_{\kappa_1}(u_\varepsilon)]= (-\Delta)^{s_0} J_{\kappa_1}(u)
	$
	in  $L^2(0,T; (H^1_0(\Omega)\cap H^{2s_1}(\Omega))^*)$ and
	\begin{align*}
&\lim\limits_{\varepsilon\to 0}	\int_{\Omega_T}\operatorname{div}(H_{\kappa_2}(|u_\varepsilon|)\nabla (-\Delta)^{-1}\mathcal{L}_\varepsilon^{1-s} [G_{\kappa_2}(u_\varepsilon)])\varphi dxdt\\&~~~~~~~~~~=
\lim\limits_{\varepsilon\to 0}	\int_{\Omega_T}\operatorname{div}(H_{\kappa_2}(|u_\varepsilon|)\nabla \varphi) (-\Delta)^{-1}\mathcal{L}_\varepsilon^{1-s} [G_{\kappa_2}(u_\varepsilon)]dxdt
\\&~~~~~~~~~~=
	\int_{\Omega_T}\operatorname{div}(H_{\kappa_2}(|u|)\nabla \varphi) (-\Delta)^{-s} [G_{\kappa_2}(u)]dxdt
	\\&~~~~~~~~~~=
	\int_{\Omega_T}\operatorname{div}(H_{\kappa_2}(|u|)\nabla (-\Delta)^{-s} [G_{\kappa_2}(u)])\varphi dxdt,
	\end{align*}
		for any $\varphi\in L^2(0,T, W^{1,\infty}_0(\Omega)\cap H^2(\Omega))$, since $\operatorname{div}(H_{\kappa_2}(|u_\varepsilon|)\nabla \varphi)\rightharpoonup \operatorname{div}(H_{\kappa_2}(|u|)\nabla \varphi)$ in $L^2(\Omega)$ and $(-\Delta)^{-1}\mathcal{L}_\varepsilon^{1-s} [G_{\kappa_2}(u_\varepsilon)]\to (-\Delta)^{-s} [G_{\kappa_2}(u)]$ in $L^2(\Omega).$
 Therefore, $u$ is a weak solution of problem \eqref{pro3}. The proof is complete.
\end{proof}

\begin{proposition} \label{step2*}Let $u_{\kappa_1}$ be a solution of problem \eqref{pro3} obtained in Proposition \ref{step2}. Then, there exists a subsequence of $\{u_{\kappa_1}\}$ converging to  a solution $u\in X$ of problem
	\begin{equation}\label{pro3*}
	\left\{
	\begin{array}
	[c]{l}%
	\partial_tu-\delta\Delta u-\operatorname{div}(H_{\kappa_2}(|u|)\nabla (-\Delta)^{-s} G_{\kappa_2}(u))+\varpi(-\Delta)^{s_0} (|u|^{m_0-1}u)=f~~\text{in }\Omega_T,\\
	u=0~~~~~~~~~~~~~~~~~~~~~~~~~~~~~~~~~~~~~~~~~~~~~~~~~~~~~~~~~~~~~~~~~~~~~~~~~~~~~~~~\text{on}
	~\partial\Omega\times (0,T),
	\\
	u(0)=u_0~~~~~~~~~~~~~~~~~~~~~~~~~~~~~~~~~~~~~~~~~~~~~~~~~~~~~~~~~~~~~~~~~~~~~~~~~~~\text{in}
	~~\Omega,
	\\
	\end{array}
	\right.
	\end{equation}
	as $\kappa_1\to0$. Moreover,
	\begin{align}\label{es66}
	|||u|^{m_0-1}u||_{L^2(0,T;H^{s_0}(\Omega))}\leq C,
	\end{align}
where constant $C$ does not depend on $u$ and $\kappa_2$.
\end{proposition}
\begin{proof} As in Proof of Proposition \ref{step2}, we have
	\begin{align}\label{es6}
	|| u_{\kappa_1}||_{L^2(0,T;H^1_0(\Omega))}+||u_{\kappa_1}||_{L^\infty(\Omega_T)}\leq C~~\forall~~\kappa_1>0,
	\end{align}
which implies
	\begin{align*}
	||\operatorname{div}(H_{\kappa_2}(|u_{\kappa_1}|)\nabla (-\Delta)^{-s} G_{\kappa_2}(u_{\kappa_1}))||_{L^2(0,T;H^{-1}(\Omega))}\leq C.
	\end{align*}
	On the other hand, we also have
	\begin{align*}
||(-\Delta)^{s_0} J_{\kappa_1}(u_{\kappa_1})||_{L^2(0,T;(H^1_0(\Omega)\cap H^{2s_0}(\Omega))^*)}\leq C ||J_{\kappa_1}(u_{\kappa_1})||_{L^2(\Omega_T)}\leq C~~\forall~~\kappa_1>0.
	\end{align*}
	Thus,
		\begin{align*}
	||\partial_tu_{\kappa_1}-\delta\Delta u_{\kappa_1}||_{L^2(0,T;(H^1_0(\Omega)\cap H^{2s_0}(\Omega))^*)}+|| u_{\kappa_1}||_{L^2(0,T;H^1_0(\Omega))}+||u_{\kappa_1}||_{L^\infty(\Omega_T)}\leq C~~\forall~~\varepsilon\in (0,1).
	\end{align*}
	As proof of Proposition \ref{step2}, there exists a subsequence of $\{u_{\kappa_1}\}$ converging to a weak solution $u\in X$ of \eqref{pro3*} in $L^2(\Omega_T)$ as $\kappa_1\to0$. Moreover, choosing $J_{\kappa_1}(u_{\kappa_1})$ as test function in \eqref{pro3} we get
\begin{align*}
||(-\Delta)^{\frac{s_0}{2}}J_{\kappa_1}(u_{\kappa_1})||_{L^2(\Omega_T)}\leq C.
\end{align*}
Letting $\kappa_1\to 0$, we find \eqref{es66}. The proof is complete.
\end{proof}

\begin{proposition} \label{step3}  Let $u_{\kappa_2}$ be a solution of problem \eqref{pro3*} obtained in Proposition \ref{step2*}.
Then, there exists a subsequence of $\{u_{\kappa_2}\}$  converging to  a solution $u\in X$ of problem
\begin{equation}\label{pro4-}
\left\{
\begin{array}
[c]{l}%
\partial_tu-\delta\Delta u-\operatorname{div}(|u|^{m_1}\nabla (-\Delta)^{-s} |u|^{m_2-1}u)+\varpi(-\Delta)^{s_0} (|u|^{m_0-1}u)=f~~\text{in }\Omega_T,\\
u=0~~~~~~~~~~~~~~~~~~~~~~~~~~~~~~~~~~~~~~~~~~~~~~~~~~~~~~~~~~~~~~~~~~~~~~~~~~~~~~\text{on}
~\partial\Omega\times (0,T),
\\
u(0)=u_0~~~~~~~~~~~~~~~~~~~~~~~~~~~~~~~~~~~~~~~~~~~~~~~~~~~~~~~~~~~~~~~~~~~~~~~~~\text{in}
~~\Omega,
\\
\end{array}
\right.
\end{equation}
	as $\kappa_2\to0$.

\end{proposition}
\begin{proof}  We have
\begin{align}\label{es67}|| u_{\kappa_2}||_{L^2(0,T;H^1_0(\Omega))}+||u_{\kappa_2}||_{L^\infty(\Omega_T)}+|||u_{\kappa_2}|^{m_0-1}
u_{\kappa_2}||_{L^2(0,T;H^{s_0}(\Omega))}\leq C~~\forall~~\kappa_2>0.
\end{align}
We will prove that
\begin{align}\label{es9}
\sup_{\kappa_2}||E_{\kappa_2}||_{L^2(0,T;H^{-1}(\Omega))} \leq C,
\end{align}
where $
E_{\kappa_2}:=\operatorname{div}(H_{\kappa_2}(|u_{\kappa_2}|)\nabla (-\Delta)^{-s}G_{\kappa_2}(u_{\kappa_2})).$
It is easy to prove \eqref{es9} in case $s\in [\frac{1}{2},1)$. So now we only consider case $s\in (0,\frac{1}{2})$.
We have for $\varphi\in L^2(0,T,H^1_0(\Omega))$,
\begin{align}\label{es68}
|\int_{\Omega_T}E_{\kappa_2}\varphi dxdt|&=|\int_{\Omega_T}(-\Delta)^{\frac{1}{2}-s}G_{\kappa_2}(u_{\kappa_2})(-\Delta)^{-\frac{1}{2}}[\operatorname{div}(H_{\kappa_2}(|u_{\kappa_2}|)\nabla\varphi)] dxdt|\\& \nonumber\leq C||(-\Delta)^{\frac{1}{2}-s}G_{\kappa_2}(u_{\kappa_2})||_{L^2(\Omega_T)}||(-\Delta)^{-\frac{1}{2}}[\operatorname{div}(H_{\kappa_2}(|u_{\kappa_2}|)\nabla\varphi)]||_{L^2(\Omega_T)}
\end{align}
By \eqref{es49} in Lemma \ref{le-H},
\begin{align}\label{es69}
||(-\Delta)^{-\frac{1}{2}}[\operatorname{div}(H_{\kappa_2}(|u_{\kappa_2}|)\nabla\varphi)]||_{L^2(\Omega_T)} &\leq C ||H_{\kappa_2}(|u_{\kappa_2}|)\nabla\varphi||_{L^2(\Omega_T)}\leq C ||\varphi||_{L^2(0,T;H^1_0(\Omega))}.
\end{align}
Since   \begin{align*}
|G_{\kappa_2}(y_1)-G_{\kappa_2}(y_2)|\leq C||y_1|^{m_0-1}y_1-|y_2|^{m_0-1}y_2| (|y_1|+|y_2|)^{m_2-m_0},
\end{align*} we have
\begin{align*}
&||(-\Delta)^{\frac{1}{2}-s}G_{\kappa_2}(u_{\kappa_2})||_{L^2(\Omega_T)}^2\\&~~~\leq C ||G_{\kappa_2}(u_{\kappa_2})||_{L^2(\Omega_T)}+C\int_{0}^{T}\int_{\Omega}\int_\Omega\frac{|G_{\kappa_2}(u_{\kappa_2})(x) -G_{\kappa_2}(u_{\kappa_2})(y) |^2}{|x-y|^{N+2(1-2s)}}dxdydt\\
&~~~\nonumber\leq C+C\int_{0}^{T}\int_{\Omega}\int_\Omega\frac{||u_{\kappa_2}(x)|^{m_0-1}u_{\kappa_2}(x) -|u_{\kappa_2}(y)|^{m_0-1}u_{\kappa_2}(y) |^2}{|x-y|^{N+2(1-2s)}}dxdydt
\\&~~~\leq C+C|||u_{\kappa_2}|^{m_0-1}u_{\kappa_2}||_{L^2(0,T;H^{s_0}(\Omega))}^2\\&~~~\leq C.
\end{align*}
Combining this with \eqref{es69} and \eqref{es68}, we get \eqref{es9}.\\
Hence, from \eqref{es9} and \eqref{es67} we have
	\begin{align*}
	||\partial_tu_{\kappa_2}-\delta\Delta u_{\kappa_2}||_{L^2(0,T;(H^1_0(\Omega)\cap H^{2s_0}(\Omega))^*)}+|| u_{\kappa_2}||_{L^2(0,T;H^1_0(\Omega))}+||u_{\kappa_2}||_{L^\infty(\Omega_T)}\leq C~\forall~\kappa_2\in (0,1).
	\end{align*}
By Lemma \ref{Simonlemma}, there exists a subsequence of $\{u_{\kappa_2}\}$ converging to $u$ in $L^1(\Omega_T)$ as $\kappa_2\to0$. Moreover, we also have $u\in X$ and
$\lim\limits_{\kappa_2\to 0}(-\Delta)^{\frac{1}{2}-s}G_{\kappa_2}(u_{\kappa_2})=(-\Delta)^{\frac{1}{2}-s}(|u|^{m_2-1}u)$, $\lim\limits_{\kappa_2\to 0} (-\Delta)^{-\frac{1}{2}}[\operatorname{div}(H_{\kappa_2}(|u_{\kappa_2}|)\nabla\varphi)]=(-\Delta)^{-\frac{1}{2}}[\operatorname{div}(|u|^{m_1})\nabla\varphi)]$~in $L^2(\Omega_T)$.

Therefore, it is easy to check that $u$ is a solution of problem \eqref{pro4-}.
\end{proof}

\begin{proof}[Proof of Proposition \ref{step4}]  Let $u_{\varpi}$ be a solution of problem \eqref{pro4-} obtained in Proposition \ref{step3}. We need to show that there exists a subsequence of $\{u_{\varpi}\}$  converging to  a solution $u\in X$ of problem \eqref{pro4}
	as $\varpi\to0$. \\ Indeed,  choosing $(|u_{\varpi}|+\eta)^{\theta-1}u_{\varpi}$ with $\theta>0$  as a test function of \eqref{pro4-},
	\begin{align*}&
	\int_{\Omega_T} |u_{\varpi}|^{m_1}\nabla (-\Delta)^{-s} (|u_{\varpi}|^{m_2-1}u_{\varpi})\nabla ((|u_{\varpi}|+\eta)^{\theta-1}u_{\varpi})\\&~~+\varpi\int_{\Omega_T} (-\Delta)^{s_0}(|u_\varpi|^{m_0-1}u_{\varpi})((|u_{\varpi}|+\eta)^{\theta-1}u_{\varpi})\leq C
	\end{align*}
	which implies
	\begin{align*}
	\int_{\Omega_T}\Gamma_\eta (v_{\varpi})(-\Delta)^{1-s}(v_{\varpi}) ++\varpi\int_{\Omega_T} (-\Delta)^{s_0}(|u_\varpi|^{m_0-1}u_{\varpi})((|u_{\varpi}|+\eta)^{\theta-1}u_{\varpi})\leq C,
	\end{align*}
	where $v_\varpi=|u_{\varpi}|^{m_2-1}u_{\varpi}$ and  $\Gamma_\eta(a)=\int_{0}^{|a|^{\frac{1}{m_2}-1}a}|b|^{m_1}(|b|+\eta)^{\theta-2}(\theta|b|+\eta)db$.\\
	By Lemma  \ref{le-Stook-Varopoulos}  and then letting $\eta\to 0$, we get
	\begin{align}\label{es78}
	\int_{\Omega_T}|(-\Delta)^{\frac{1-s}{2}}\left(|u_\varpi|^{\frac{\gamma+\theta}{2}-1}u_\varpi\right)|^2+
\varpi\int_{\Omega_T}|(-\Delta)^{\frac{s_0}{2}}(|u_\varpi|^{\frac{m_0+\theta}{2}-1}u_{\varpi})|^2\leq C,
	\end{align}
	with $\gamma=m_1+m_2$.
Thus, for any $\theta\in (0,1)$
	\begin{align*}
	&|| u_\varpi||_{L^2(0,T;H^1_0(\Omega))}+||u_\varpi||_{L^\infty(\Omega_T)}+|||u_\varpi|^{\frac{\gamma+\theta}{2}-1}u_\varpi||_{L^2(0,T;H^{1-s}(\Omega))}\leq C~~\forall~~\varpi>0.
	\end{align*}
	By Lemma \ref{le-f} below, we have
	\begin{align*}
	&||\operatorname{div}(|u_\varpi|^{m_1}\nabla (-\Delta)^{-s}(|u_\varpi|^{m_2-1}u_\varpi)||_{L^2(0,T;(H^1_0(\Omega)\cap W^{2-\vartheta,r}(\Omega)))^{*}}\\~~~~~~&\leq C \left(\int_{0}^{T}|||u_\varpi|^{\gamma-1}u_\varpi||^2_{H^{(1-2s)^+}(\Omega)} dt\right)^{1/2}\\&\leq C,
	\end{align*}
	for some $r>1,\vartheta\in (0,1)$.
Hence,
\begin{align*}
||\partial_tu_{\varpi}-\delta\Delta u_{\varpi}||_{L^2(0,T;(H^1_0(\Omega)\cap W^{2-\vartheta,r}(\Omega)))^{*}}+|| u_{\varpi}||_{L^2(0,T;H^1_0(\Omega))}+||u_{\varpi}||_{L^\infty(\Omega_T)}\leq C~\forall~\varpi\in (0,1).
\end{align*}
for some $r>1,\vartheta\in (0,1)$.
By Lemma \ref{Simonlemma}, there exists a subsequence of $\{u_{\varpi}\}$ converging to $u$ in $L^1(\Omega_T)$ as $\varpi\to0$. Moreover, we have for $\varphi\in H^1_0(\Omega)$
\begin{align*}
\left|\varpi\int_{\Omega_T}(-\Delta)^{s_0}(|u_\varpi|^{m_0-1}u_{\varpi})\varphi\right|&\leq C \varpi ||(-\Delta)^{\frac{s_0}{2}}(|u_\varpi|^{m_0-1}u_{\varpi})||_{L^2(\Omega_T)} ||(-\Delta)^{\frac{s_0}{2}}\varphi||_{L^2(\Omega_T)}\\&\overset{\eqref{es78}}\leq C\sqrt{\varpi}||\varphi||_{H^1_0(\Omega)}\to 0~~\text{as}~~\varpi\to 0.
\end{align*}
 Therefore, it is easy to check that $u$ is a solution of problem \eqref{pro4} and belongs to $X$. The proof is complete.
\end{proof}
In proof of Proposition \ref{step4}, we have used the following  basic lemma.
\begin{lemma} \label{le-f}  There exists $\vartheta=\vartheta(s,m_1,m_2)\in (0,1/2)$ and $r=r(s,m_1,m_2,N)\in (2,\infty)$ such that
	\begin{align}\label{es32*}
	||\operatorname{div}(|v|^{m_1}\nabla (-\Delta)^{-s}(|v|^{m_2-1}v)||_{(H^1_0(\Omega)\cap W^{2-\vartheta,r}(\Omega))^{*}}\leq C |||v|^{\gamma-1}v||_{H^{(1-2s)^+}(\Omega)}
	\end{align}
	for all $|v|^{\gamma-1}v\in H^{(1-2s)^+}(\Omega).$
\end{lemma}

\begin{proof}[Proof of Lemma \ref{le-f} ] It is easy to prove  \eqref{es32*} in case $s\in [\frac{1}{2},1)$. Thus, we only consider case $s\in (0,\frac{1}{2})$. Let $\beta\in (s,1/2)$ be such that
	\begin{align}\label{es31}
	\frac{(1-2s)m_1}{\gamma}=1-2\beta,~~~
	\frac{(1-2s)m_2}{\gamma}=2(\beta-s).
	\end{align}
	Since for $a>0$ and $b\in (0,1)$
	\begin{align*}
	||y_1|^{a-1}y_1-|y_2|^{a-1}y_2|\geq C ||y_1|^{ab-1}y_1-|y_1|^{ab-1}y_1|^{\frac{1}{b}} ~\forall~y_1,y_2\in \mathbb{R},
	\end{align*}
	thus,
\begin{align}\label{es76}
|||u|^{m_1-1}u||_{H^{1-2\beta}(\Omega)}&\leq C |||u|^{m_1-1}u||_{W^{1-2\beta,\frac{2\gamma}{m_1}}(\Omega)}\leq  C |||u|^{\gamma-1}u||_{H^{1-2s}(\Omega)}^{\frac{m_1}{\gamma}},
\end{align}
and
\begin{align}\label{es77}
|||u|^{m_2-1}u||_{H^{2(\beta-s)}(\Omega)}&\leq  C |||u|^{\gamma-1}u||_{H^{1-2s}(\Omega)}^{\frac{m_2}{\gamma}}.
\end{align}
Therefore, for $\varphi\in H^1_0(\Omega)\cap W^{2,\infty}(\Omega)$,
	\begin{align*}
	&\left|\int_{\Omega} \operatorname{div}(|u|^{m_1}\nabla (-\Delta)^{-s}(|u|^{m_2-1}u)\varphi\right|\\&=\left|\int_{\Omega}(-\Delta)^{\beta-s}(|u|^{m_2-1}u)(-\Delta)^{-\beta}\left[\operatorname{div}(|u|^{m_1}\nabla \varphi)\right] dx\right|\\&\leq ||(-\Delta)^{\beta-s}(|u|^{m_2-1}u)||_{L^2(\Omega)} ||(-\Delta)^{-\beta}\left[\operatorname{div}(|u|^{m_1}\nabla \varphi)\right]||_{L^2(\Omega)}
	\\&\overset{\eqref{es49} ~ \text{in Lemma}~ \ref{le-H}}\leq ||(-\Delta)^{\beta-s}(|u|^{m_2-1}u)||_{L^2(\Omega)} |||u|^{m_1}\nabla \varphi||_{H^{1-2\beta}(\Omega)}\\&
	\leq C |||u|^{m_2-1}u||_{H^{2(\beta-s)}(\Omega)}|||u|^{m_1}||_{H^{1-2\beta}(\Omega)}||\varphi||_{W^{2-2\beta,\infty}(\Omega)}
	\\&
	\leq C  |||u|^{\gamma-1}u||_{H^{1-2s}(\Omega)}||\varphi||_{W^{2-2\beta,\infty}(\Omega)}
	\end{align*}
	which implies \eqref{es32*}. The proof is complete.
\end{proof}

\section{Existence of weak solutions via approximation}\label{sec.proof.th1}
In this section, we prove Theorem \ref{mainthm} by using the approximate problems of the preceding section.  Let $u_\delta=u \in L^\infty(\Omega_T)\cap C(0,T;L^r(\Omega))\cap L^2(0,T;H^1_0(\Omega))$  be a solution of \eqref{pro4} for all $r<\infty$. Set $M=||u_0||_{L^1(\Omega)}+||f||_{L^1(\Omega_T)}$.  The proof of the theorem will be obtained from Lemma \ref{le-1}, \ref{le-2}, \ref{le-3}, \ref{le-4},\ref{le-5} and \ref{le-6} with $u=u_{\delta}$. The complete proof is at Lemma \ref{le-7}.

 \smallskip
 \begin{lemma}[Estimates for $L^1$-data] \label{le-1} There hold,
 	\begin{align}\label{main-es1}
 	||u||_{L^\infty(0,T,L^1(\Omega))}\leq ||u_0||_{L^1(\Omega)}+||f||_{L^1(\Omega_T)},
 	\end{align}
 	and
 	\begin{align}\label{main-es1"}
 	||u^\pm||_{L^\infty(0,T,L^1(\Omega))}\leq ||u_0^\pm||_{L^1(\Omega)}+||f^\pm||_{L^1(\Omega_T)}.
 	\end{align}
 	In particular, if $u_0,f\geq 0$, then $u\geq 0$,
 	\begin{align} \label{main-es1'}
 	&1_{s<1-\frac{N}{2}}||u||_{L^{\gamma+\frac{2(1-s)}{N},\infty}(\Omega_T)}+1_{s=1-\frac{N}{2}}||u||_{L^{\gamma+1-\frac{1}{l},\infty}(\Omega_T)}+1_{s>1-\frac{N}{2}}||u||_{L^{\gamma+1,\infty}(\Omega_T)}\\&\nonumber\leq C 1_{s<1-\frac{N}{2}}M^{\frac{N+2(1-s)}{\gamma N-2(1-s)}}+C1_{s=1-\frac{N}{2}}M^{\frac{2l}{l(\gamma+1)-1}}+C1_{s>1-\frac{N}{2}}M^{\frac{2}{\gamma+1}},
 	\end{align}
 	for all $l>1$,
 	and
 	\begin{align}\label{es45}
 	\int_{0}^T\int_\Omega|(-\Delta)^{\frac{1-s}{2}}\left(\frac{|u|^{\frac{m_1+m_2}{2}+\theta-1}u}{|u|^{2\theta}+1}\right)|^2dxdt\leq CM~~ \forall~~\theta>0.
 	\end{align}
 \end{lemma}

\smallskip

\noindent {\sl Proof.}
 Choosing $T_k(u):=\min\{|u|,k\}\text{sgn}(u)$ as test function of \eqref{pro4},
\begin{align}\label{es7}
 ||\overline{T}_k(u)||_{L^\infty(0,T,L^1(\Omega))}+ \int_{\Omega_T}|u|^{m_1}\nabla (-\Delta)^{-s}(|u|^{m_2-1}u)\nabla T_k(u)\leq kM,
\end{align}
with $\overline{T}_k(u)=\int_{0}^{u}T_k(a)da.$
Since $\lim\limits_{k\to 0}\overline{T}_k(u) k^{-1}=u$, we get \eqref{main-es1}. Similarly,  choosing $T_k(u)^+:=\min\{|u|,k\}1_{u\geq 0}$ as test function of Problem \eqref{pro4} then we will get
\begin{align}\label{main-es1''}
||u^+||_{L^\infty(0,T,L^1(\Omega))}\leq ||u_0^+||_{L^1(\Omega)}+||f^+||_{L^1(\Omega_T)},
\end{align}
which implies \eqref{main-es1"}. In particular, if $u_0,f\geq 0$, then $u\geq 0$.

\textbf{1.  Proof of \eqref{main-es1'}.} First, we prove that
\begin{align}
\label{es44}\int_{0}^{T}\int_\Omega |(-\Delta)^{\frac{1-s}{2}}(\eta_k(|u|^{m_2-1}u))|^2dxdt\leq  Ck^{\frac{m_2-m_1}{m_2}} M~~\forall~~k>0,
\end{align}
where $\eta_k(s)=k\eta(s/k)$,   $\eta$ is  a smooth function in $\mathbb{R}$ such that $\eta(s)=0$ if $|s|\leq 1/2$, $|\eta'(s)|=1$ if $1\leq |s|\leq 2$ and $|\eta(s)|=3$ if $|s|> 3$.

Set $v=|u|^{m_2-1}u$,
we have from \eqref{es7} that
\begin{align}
 \int_{\Omega_T}(-\Delta)^{1-s}(v) (|T_k(|v|^{\frac{1}{m_2}-1}v)|^{m_1}T_k(|v|^{\frac{1}{m_2}-1}v))\leq C kM~~\forall k>0.
\end{align}
It is equivalent to
\begin{align}\label{es35}
\int_{\Omega_T}(-\Delta)^{1-s}(v) T_k(|v|^{\frac{m_1+1-m_2}{m_2}}v)\leq C k^{\frac{1}{m_1+1}}M~~\forall ~~k>0.
\end{align}
Let $V(.)=V(t,.)$ be a unique solution of the extended problem
	\begin{equation*}
\left\{
\begin{array}
{ll}%
\operatorname{div}_{x,y}(y^{1-2(1-s)}\nabla_{x,y} V)=0~~&\text{in }\Omega\times(0,\infty),\\
V=0 &\text{on}
~~\partial\Omega\times (0,\infty),
\\
V(x,0)=v(x) &\text{in}
~~\Omega,
\\
\end{array}
\right.
\end{equation*}
For the equivalence of this problem with the original problem with the spectral Laplacian see for instance \cite{Cabre-Tan, pqrv1, pqrv2}.
We have
\begin{align}\label{es38}
\int_{\Omega}\int_0^\infty y^{1-2(1-s)}\nabla_{x,y} V\nabla_{x,y} \varphi dydx=c_s\int_{\Omega}(-\Delta)^{1-s}(v) \varphi(0)dx,\end{align}
for any $\varphi\in H^1_0(\Omega\times (0,\infty),d\omega)$ with $d\omega= y^{1-2(1-s)}dydx$.

From this and \eqref{es35} we deduce for all $k>0$
\begin{align*}
\int_{0}^{T}\int_{\Omega}\int_0^\infty y^{1-2(1-s)}\nabla_{x,y} V\nabla T_k(|V|^{\frac{m_1+1-m_2}{m_2}}V) dydxdt\leq Ck^{\frac{1}{m_1+1}}M~\forall~k>0.\end{align*}
So,
\begin{align}\label{es41}
\int_{0}^{T}\int_{\Omega}\int_0^\infty 1_{k\leq |V|\leq 2k}y^{1-2(1-s)}|\nabla_{x,y} V |^2dydx dt\leq Ck^{\frac{m_2-m_1}{m_2}} M~~\forall~~k>0.\end{align}

Let $W(.)=W(t,.)$ be a unique solution of the extended problem
\begin{equation*}
\left\{
\begin{array}
{ll}%
\operatorname{div}_{x,y}(y^{1-2(1-s)}\nabla_{x,y} W)=0~~&\text{in }\Omega\times(0,\infty),\\
W=0 &\text{on}
~~\partial\Omega\times (0,\infty),
\\
W(x,0)=\eta_k(v(x)) &\text{in}
~~\Omega,
\\
\end{array}
\right.
\end{equation*}

Since $\operatorname{div}_{x,y}(y^{1-2(1-s)}\nabla_{x,y} \eta_k(V))=\eta_k''(V)y^{1-2(1-s)}|\nabla_{x,y}V|^2$,
\begin{align*}
&\int_\Omega\int_{0}^{\infty}y^{1-2(1-s)}|\nabla_{x,y}W|^2dydx\\&=\int_\Omega\int_{0}^{\infty}y^{1-2(1-s)}\nabla_{x,y} \eta_k(V)\nabla_{x,y}Wdydx+\int_\Omega\int_{0}^{\infty}\eta_k''(V)y^{1-2(1-s)}|\nabla_{x,y}V|^2 Wdydx.
\end{align*}
Using H\"older's inequality and the fact that $|W|\leq ||W(.,0)||_{L^\infty(\Omega)}\leq k||\eta||_{L^\infty(\Omega)}$ yields
\begin{align*}
\int_{0}^{T}\int_\Omega\int_{0}^{\infty}y^{1-2(1-s)}|\nabla_{x,y}W|^2dydx&\leq C\int_{0}^{T}\int_\Omega\int_{0}^{\infty}1_{k/2\leq v\leq 3k}y^{1-2(1-s)}|\nabla_{x,y} V|^2dydx\\&\overset{\eqref{es41}}\leq Ck^{\frac{m_2-m_1}{m_2}} M.
\end{align*}
From this and
\begin{align*}
\int_{0}^{T}\int_\Omega\int_{0}^{\infty}y^{1-2(1-s)}|\nabla_{x,y}W|^2dydx=c_s\int_{0}^{T}\int_\Omega |(-\Delta)^{\frac{1-s}{2}}(\eta_k(v))|^2dxdt,
\end{align*}  we find \eqref{es44}. By \eqref{L^p-b-frac1}, \eqref{L^p-b-frac2}, \eqref{L^p-b-frac3} in Lemma \ref{Lpes}, we have
\begin{align*}
&\int_\Omega |(-\Delta)^{\frac{1-s}{2}}(\eta_k(|u|^{m_2-1}u))|^2dxdt
\\&\geq C 1_{s>1-\frac{N}{2}}||\eta_k(|u|^{m_2-1}u))||_{L^{\frac{2N}{N-2(1-s)}}(\Omega)}^2+C 1_{s=1-\frac{N}{2}}||\eta_k(|u|^{m_2-1}u))||_{BMO(\Omega)}^2\\&\nonumber+C 1_{s<1-\frac{N}{2}}||\eta_k(|u|^{m_2-1}u))||_{L^\infty(\Omega)}^2
\\&\nonumber\geq C k^21_{s>1-\frac{N}{2}}||1_{|u|\geq k^{1/m_2}}||_{L^{\frac{2N}{N-2(1-s)}}(\Omega)}^2+C k^21_{s=1-\frac{N}{2}}||1_{|u|\geq k^{1/m_2}}||_{BMO(\Omega)}^2\\&\nonumber+C k^2 1_{s<1-\frac{N}{2}}||1_{|u|\geq k^{1/m_2}}||_{L^\infty(\Omega)}^2.
\end{align*}
Combining this with \eqref{es44}, we deduce
\begin{align}\label{es8}
  &1_{s>1-\frac{N}{2}}\int_{0}^{T}||1_{|u|\geq k}||_{L^{\frac{2N}{N-2(1-s)}}(\Omega)}^2dt+ 1_{s=1-\frac{N}{2}}\int_{0}^{T}||1_{|u|\geq k}||_{BMO(\Omega)}^2dt\\&\nonumber + 1_{s<1-\frac{N}{2}}\int_{0}^{T}||1_{|u|\geq k}||_{L^\infty(\Omega)}^2dt\leq Ck^{-m_1-m_2}M~~\forall~k>0.
\end{align}
\\Case $s>1-\frac{N}{2}$,
\begin{align*}
|\{|u|>k\}|&= \int_{0}^{T}\left[\int_{\Omega}1_{|u|\geq k}\right]^{\frac{2(1-s)}{N}} \left[\int_{\Omega}1_{|u|\geq k}\right]^{\frac{N-2(1-s)}{N}}\\&\leq
\left[\sup_{t\in (0,T)}\int_{\Omega}1_{|u|\geq k}\right]^{\frac{2(1-s)}{N}}\int_{0}^{T}\left[\int_{\Omega}1_{|u|\geq k}\right]^{\frac{N-2(1-s)}{N}}
\\&\overset{\eqref{es8},\eqref{main-es1}}\leq C
\left[k^{-1}M\right]^{\frac{2(1-s)}{N}} k^{-\gamma}M
\\&\leq Ck^{-\gamma-\frac{2(1-s)}{N}} M^{\frac{N+2(1-s)}{N}}.
\end{align*}
Case $s=1-\frac{N}{2}$,~~for any $l>1$
\begin{align*}
|\{|u|>k\}|&= \int_{0}^{T}\left[\int_{\Omega}1_{|u|\geq k}\right]^{1-\frac{1}{l}} \left[\int_{\Omega}1_{|u|\geq k}\right]^{\frac{1}{l}}
\overset{\eqref{es8},\eqref{main-es1}}
\leq Ck^{-\gamma-1+\frac{1}{l}} M^{2-\frac{1}{l}}.
\end{align*}
Case $s<1-\frac{N}{2}$,
\begin{align*}
|\{|u|>k\}|&\leq
\left[\sup_{t\in (0,T)}\int_{\Omega}1_{|u|\geq k}\right]\left[\int_{0}^{T}\sup_{x\in \Omega}1_{|u|\geq k}\right]
\overset{\eqref{es8},\eqref{main-es1}}
\leq Ck^{-\gamma-1} M^{2}.
\end{align*}
Therefore, we get \eqref{main-es1'}.\medskip\\\\
\textbf{2.  Proof of \eqref{es45}.}
Let $\chi$ be a smooth function in $\mathbb{R}^+$ such that $\chi(s)=1$ if $|s|\leq 1$,  and $\chi(s)=0$ if $|s|> 2$.  Set $\psi_{j}(v)=\left[\chi(2^{-j}v)-\chi(2^{-j+1}v)\right] \left(\frac{|v|^{\frac{\gamma}{2m_2}+\theta-1}v}{|v|^{2\theta}+1}\right)$.
Let $U(.)=U(t,.)$ be a unique solution of the extended problem
\begin{equation*}
\left\{
\begin{array}
{ll}%
\operatorname{div}_{x,y}(y^{1-2(1-s)}\nabla_{x,y} U)=0~~&\text{in }\Omega\times(0,\infty),\\
U=0 &\text{on}
~~\partial\Omega\times (0,\infty),
\\
U(x,0)=\psi_j(v(x)) &\text{in}
~~\Omega.
\\
\end{array}
\right.
\end{equation*}

As proof of \eqref{es44}, we have
\begin{align*}
&\int_\Omega\int_{0}^{\infty}y^{1-2(1-s)}|\nabla_{x,y}U|^2dydx\\&\leq C\int_\Omega\int_{0}^{\infty}y^{1-2(1-s)}|\nabla_{x,y} \psi_j(V)|^2dydx+C\int_\Omega\int_{0}^{\infty}|\psi_j^{''}(V)|y^{1-2(1-s)}|\nabla_{x,y}V|^2 |U|dydx.
\end{align*}
Since  $|U|\leq ||\psi_j(v(x)) ||_{L^\infty(\Omega)}\leq C\frac{(2^j)^{\frac{\gamma}{2m_2}+\theta}}{(2^j+1)^{2\theta}}$ and
\begin{align*}
|\psi_j'(V)|\leq C1_{2^{j-1}\leq |V|\leq 2^{j}} \frac{(2^j)^{\frac{\gamma}{2m_2}-1+\theta}}{(2^j+1)^{2\theta}},~~
|\psi_j''(V)|\leq C1_{2^{j-1}\leq |V|\leq 2^{j}} \frac{(2^j)^{\frac{\gamma}{2m_2}-2+\theta}}{(2^j+1)^{2\theta}}.
\end{align*}
so,
\begin{align*}
&\int_{0}^{T}\int_\Omega\int_{0}^{\infty}y^{1-2(1-s)}|\nabla_{x,y}U|^2dydx dt\\&\leq C\frac{(2^j)^{\frac{\gamma}{m_2}-2+2\theta}}{(2^j+1)^{4\theta}}\int_\Omega\int_{0}^{\infty}1_{2^{j-1}\leq |V|\leq 2^{j}}y^{1-2(1-s)}|\nabla_{x,y} V|^2dydx dt\\&\overset{\eqref{es41}}\leq C\frac{(2^j)^{\frac{\gamma}{m_2}-2+2\theta}}{(2^j+1)^{4\theta}} (2^j)^{\frac{m_2-m_1}{m_2}} M
\leq C\frac{(2^j)^{2\theta}}{(2^j+1)^{4\theta}} M.
\end{align*}
Thus,
\begin{align}\label{es46}
\left(\int_{0}^{T}\int_\Omega |(-\Delta)^{\frac{1-s}{2}}\psi_j(v)|^2dx dt\right)^{1/2}\leq C\frac{(2^j)^{\theta}}{(2^j+1)^{2\theta}} M.
\end{align}
Since $
\sum_{j=-k}^{j=k}\psi_{j}(v)\to \frac{|v|^{\frac{\gamma}{2m_2}+\theta-1}v}{|v|^{2\theta}+1}~~\text{as}~~k\to \infty,$
we derive  from \eqref{es46} that
\begin{align*}
\left(\int_{0}^T\int_\Omega|(-\Delta)^{\frac{1-s}{2}}\left(\frac{|v|^{\frac{\gamma}{2m_2}+\theta-1}v}{|v|^{2\theta}+1}\right)|^2dxdt\right)^{1/2}&\leq \sum_{j=-\infty}^{\infty}\left(\int_{0}^{T}\int_\Omega |(-\Delta)^{\frac{1-s}{2}}\psi_j(v)|^2dx dt\right)^{1/2}\\&\leq \sum_{j=-\infty}^{\infty}C\frac{(2^j)^{\theta}}{(2^j+1)^{2\theta}} M
\\&\leq CM
\end{align*}
which implies \eqref{es45}.

\medskip

\begin{lemma}\label{le-2}For $p\in (1,\infty)$
	\begin{align}\label{main-es2}
	\frac{d}{dt}\int_\Omega |u(t)|^p&+ \delta p(p-1)\int_{\Omega} |u|^{p-2}|\nabla u|^2 +\frac{4m_2p(p-1)}{(\gamma+p-1)^2}\int_{\Omega}| (-\Delta)^{\frac{1-s}{2}}(|u|^{\frac{\gamma+p-1}{2}-1}u)|^2\\&\leq p\left|\int_{\Omega} f|u|^{p-2}u\right|\nonumber.
	\end{align}
 In particular, 	{(i).}
\begin{align}\label{main-es2'}
	\int_\Omega |u(t)|^p&+ \delta p(p-1)\int_{0}^{t}\int_{\Omega} |u|^{p-2}|\nabla u|^2 +\frac{4m_2p(p-1)}{(\gamma+p-1)^2}\int_{0}^{t}\int_{\Omega}| (-\Delta)^{\frac{1-s}{2}}(|u|^{\frac{\gamma+p-1}{2}-1}u)|^2\\&\leq \int_\Omega |u_0|^p+p\int_{0}^{t}\int_{\Omega} |f||u|^{p-1},\nonumber
	\end{align}
	for all $t\in (0,T).$
	
{(ii).} \begin{align}\label{main-es2''"}
	\int_\Omega |u(t)|^{\gamma+1}&+ \delta  \int_{0}^{t}\int_{\Omega} |u|^{\gamma-1}|\nabla u|^2 +\int_{0}^{t}\int_{\Omega}| (-\Delta)^{\frac{1-s}{2}}(|u|^{\gamma-1}u)|^2
	\\&\leq C \int_\Omega |u_0|^{\gamma+1}+ C\int_{0}^{t}\int_{\Omega} |(-\Delta)^{-\frac{1-s}{2}}f|^2\nonumber,
	\end{align}
	for all $t\in (0,T).$
\end{lemma}

\begin{proof} 1. For $p\in (1,\infty)$, choosing $(|u|+\varepsilon)^{p-2}u$ as test function of \eqref{pro4},
	\begin{align*}
	&\frac{d}{dt}\int_\Omega\int_{0}^{u(t)}(|a|+\varepsilon)^{p-2}a da dx+ \delta\int_{\Omega} \nabla u\nabla [(|u|+\varepsilon)^{p-2}u]\\&+\int_{\Omega} |u|^{m_1}\nabla (-\Delta)^{-s}(|u|^{m_2-1}u)\nabla \left[(|u|+\varepsilon)^{p-2}u\right]\leq \int_{\Omega} f(|u|+\varepsilon)^{p-2}u
	\end{align*}
	for all $t\in (0,T)$. By Lemma \eqref{le-Stook-Varopoulos} and Corollary \ref{stook-var-in-coro} and  then Letting $\varepsilon\to 0$,
	we have
	\begin{align*}
	\liminf_{\varepsilon\to 0}\int_{\Omega} |u|^{m_1}\nabla (-\Delta)^{-s}(|u|^{m_2-1}u)\nabla \left[(|u|+\varepsilon)^{p-2}u\right]\geq \frac{4m_2(p-1)}{(\gamma+p-1)^2}\int_{\Omega}|(-\Delta)^{\frac{1-s}{2}}(|u|^{\frac{\gamma+p-1}{2}-1}u)|^2
	\end{align*}

	Thus, we find \eqref{main-es2} and \eqref{main-es2'}.\medskip\\
	2. Applying \eqref{main-es2} to $p=\gamma+1$, we have
	\begin{align}\label{es23}
	\frac{d}{dt}\int_\Omega |u(t)|^{\gamma+1}&+ C\delta \int_{\Omega} |u|^{\gamma-1}|\nabla u|^2 +C\int_{\Omega}| (-\Delta)^{\frac{1-s}{2}}(|u|^{\gamma-1}u)|^2\\&\leq (\gamma+1)\left(\int_{\Omega} |(-\Delta)^{-\frac{1-s}{2}}f|^2\right)^{\frac{1}{2}}\left(\int_{\Omega}|(-\Delta)^{\frac{1-s}{2}}(|u|^{\gamma-1}u)|^2\right)^{\frac{1}{2}}\nonumber.
	\end{align}
	So, using  H\"older's inequality we derive \eqref{main-es2''"}. \end{proof}

\begin{lemma}\label{le-3}\begin{align}\label{main-es3}
	||u||_{L^\infty(\Omega_T)}\leq ||u_0||_{L^\infty(\Omega)}+T||f||_{L^\infty(\Omega_T)}.
	\end{align}
\end{lemma}
\begin{proof} From \eqref{main-es2}, we get
	\begin{align*}
	\int_\Omega |u(t)|^p\leq \int_\Omega |u_0|^p+p\int_{0}^{T}\int_{\Omega} |f||u|^{p-1}.
	\end{align*}
	Fix $\lambda>T$, we have
	\begin{align*}
	&(1-T\lambda^{-\frac{p}{p-1}})\sup_{t\in (0,T)}\int_\Omega |u(t)|^p+\lambda^{-\frac{p}{p-1}}\int_{\Omega_T} |u|^p\leq \sup_{t\in (0,T)}\int_\Omega |u(t)|^p\\&~~~~~~~~~\leq \int_\Omega |u_0|^p+p\int_{\Omega_T} |f||u|^{p-1}\\&
	\overset{\text{H\"older's inequality}}\leq \int_\Omega |u_0|^p+\lambda^{p}\int_{\Omega_T} |f|^p+\lambda^{-\frac{p}{p-1}}\int_{\Omega_T}|u|^{p}.
	\end{align*}
	So
	\begin{align*}
	&(1-T\lambda^{-\frac{p}{p-1}})^{1/p}\sup_{t\in (0,T)}\left[\int_\Omega |u(t)|^p\right]^{1/p}\leq \left[\int_\Omega |u_0|^p+\lambda^{p}\int_{\Omega_T} |f|^p\right]^{1/p}.
	\end{align*}
	Letting $p\to\infty$,
	\begin{align*}
	||u||_{L^\infty(\Omega_T)}\leq ||u_0||_{L^\infty(\Omega)}+\lambda||f||_{L^\infty(\Omega_T)}~~\forall \lambda>T,
	\end{align*}
	which implies \eqref{main-es3}.\medskip\\
\end{proof}
\begin{lemma}\label{le-4} If $m_2=1$,
	\begin{align}\label{main-es4}
	&\frac{1}{2}	\int_\Omega |(-\Delta)^{-\frac{s}{2}}u(t)|^2dx+\delta\int_{0}^{t}\int_{\Omega}|(-\Delta)^{\frac{1-s}{2}}u|^2+\int_0^t\int_{\Omega}|u|^{m_1}|\nabla (-\Delta)^{-s}(u)|^{2}dxdt\\&\leq \frac{1}{2}\int_\Omega |(-\Delta)^{-\frac{s}{2}}u_0|^2dx+\left(\int_{0}^{t}\int_\Omega|(-\Delta)^{-\frac{s}{2}}f|^2dxdt\right)^{1/2}
\left(\int_{0}^{t}\int_\Omega|(-\Delta)^{-\frac{s}{2}}u|^2dxdt\right)^{1/2}\nonumber
	\end{align}
	for all $t\in (0,T)$.
\end{lemma}
\begin{proof}Choosing $(-\Delta)^{-s}u$ as test function of \eqref{pro4}, we find \eqref{main-es4}.
\end{proof}
\begin{lemma}\label{le-5}
	Assume $f=0$. Let  $q_0\geq 1$ be such that  $N(\gamma-1)+2q_0(1-s)>0$. Then, there holds \begin{align}\label{main-es5}
	||u(t)||_{q}\leq  C ||u_0||_{L^{q_0}(\Omega)}^{\frac{N(\gamma-1)\frac{q_0}{q}+2q_0(1-s)}{N(\gamma-1)+2q_0(1-s)}}t^{-\frac{(1-\frac{q_0}{q})N}{N(\gamma-1)+2q_0(1-s)}}~~\forall ~q\in [q_0,\infty].
	\end{align}
\end{lemma}
\begin{proof}
Applying \eqref{main-es2} to $f=0$,
\begin{align}
\frac{d}{dt}\int_\Omega |u(t)|^p +\frac{4m_2p(p-1)}{(\gamma+p-1)^2}\int_{\Omega}| (-\Delta)^{\frac{1-s}{2}}(|u|^{\frac{\gamma+p-1}{2}})|^2\leq 0.
\end{align}
By \eqref{L^p-b-frac1}, \eqref{L^p-b-frac2} and \eqref{L^p-b-frac3} in Lemma \ref{Lpes}, we have
\begin{align}\label{es24}
\frac{d}{dt}\int_\Omega |u(t)|^p&\leq- C\frac{4m_2p(p-1)}{(\gamma+p-1)^2}\left[1_{s>1-\frac{N}{2}}||u||_{L^{\frac{(\gamma+p-1)N}{N-2(1-s)}}(\Omega)}^{\gamma+p-1}\right.\\&\left.~~~+1_{s=1-\frac{N}{2}}|||u|^{\frac{\gamma+p-1}{2}}||_{BMO(\Omega)}^2+1_{s<1-\frac{N}{2}}||u||_{L^\infty(\Omega)}^{\gamma+p-1}\right]\nonumber\\&\leq -C\left[1_{s>1-\frac{N}{2}}||u||_{L^{\frac{(\gamma+p-1)N}{N-2(1-s)}}(\Omega)}^{\gamma+p-1}+1_{s=1-\frac{N}{2}}|||u|^{\frac{\gamma+p-1}{2}}||_{BMO(\Omega)}^2+1_{s<1-\frac{N}{2}}||u||_{L^\infty(\Omega)}^{\gamma+p-1}\right]\nonumber
\end{align}
for all $p>p_0>1$, since $
\frac{4m_2p(p-1)}{(\gamma+p-1)^2}\geq C_{p_0}$ for all $p>p_0$.\\
Let $q_0\geq 1$ be such that $N(\gamma-1)+2q_0(1-s)>0$.\medskip\\
It is enough to prove \eqref{main-es5} with  $||u_0||_{L^{q_0}(\Omega)}=1$. By \eqref{main-es2'}, we have $||u(t)||_{L^{q_0}(\Omega)}\leq 1$.

Assume $s>1-\frac{N}{2}$. We have  from \eqref{es24} that
\begin{align}\label{es25}
\frac{d}{dt}\int_\Omega |u(t)|^p\leq- C||u||_{L^{\frac{(\gamma+p-1)N}{N-2(1-s)}}(\Omega)}^{\gamma+p-1}.
\end{align}
Let  $q_0\leq q<p$. Clearly,
\begin{align*}
\frac{\gamma-1}{q}+\frac{2(1-s)}{N}>0,~~\beta=\frac{\frac{\gamma-1}{p}+\frac{2(1-s)}{N}}{\frac{\gamma+p-1}{q}-1+\frac{2(1-s)}{N}}\in (0,1).
\end{align*}
By interpolation inequality,
\begin{align*}
||u||_{L^{\frac{(\gamma+p-1)N}{N-2(1-s)}}(\Omega)}^{\gamma+p-1}\geq ||u||_{L^p(\Omega)}^{\frac{\gamma+p-1}{1-\beta}}||u||_{L^q(\Omega)}^{-\frac{\beta(\gamma+p-1)}{1-\beta}}.
\end{align*}
Thus,
\begin{align*}
\frac{d}{dt}\int_\Omega |u(t)|^p&\leq-C||u||_{L^p(\Omega)}^{\frac{\frac{\gamma+p-1}{q}-1+\frac{2(1-s)}{N}}{\frac{1}{q}-\frac{1}{p}}}||u||_{L^q(\Omega)}^{-\frac{\frac{\gamma-1}{p}+\frac{2(1-s)}{N}}{\frac{1}{q}-\frac{1}{p}}}.
\end{align*}
Set $F_{r}(t)=||u(t)||_{L^r(\Omega)}^{-1}$ for all $r\in (1,\infty]$, $t\mapsto F_{r}(t)$ is nondecreasing and
\begin{align*}
\frac{d}{dt}F_{p}(t)^{-p}\leq -C F_{p}(t)^{-\frac{\frac{\gamma+p-1}{q}-1+\frac{2(1-s)}{N}}{\frac{1}{q}-\frac{1}{p}}} F_{q}(t)^{\frac{\frac{\gamma-1}{p}+\frac{2(1-s)}{N}}{\frac{1}{q}-\frac{1}{p}}}.
\end{align*}
Leads to
\begin{align*}
\frac{d}{dt}F_p(t)^{\frac{p(N(\gamma-1)+2q(1-s))}{N(p-q)}}\geq C\frac{N(\gamma-1)+2q(1-s)}{N(p-q)}F_{q}(t)^{\frac{q(N(\gamma-1)+2p(1-s))}{N(p-q)}}.
\end{align*}

Now we apply this to $p=p_k=2^{k}q_0$ and $q=p_{k-1}=2^{k-1}q_0$
\begin{align*}
F_{p_k}(t)^{\frac{2(N(\gamma-1)+2^{k}q_0(1-s))}{N}}\geq \int_{0}^{t}C\frac{N(\gamma-1)+2^k q_0(1-s)}{N2^{k-1}q_0}F_{p_{k-1}}(\tau)^{\frac{(N(\gamma-1)+2^{k+1}q_0(1-s))}{N}}d\tau~~\forall~~t>0.
\end{align*}
Thus,
\begin{align}\label{es27}
F_{p_k}(t)\geq c_k t^{\vartheta_k}~~\forall~t>0
\end{align}
where $c_\kappa,\vartheta_k$ satisfy  $
c_0=1,\vartheta_0=0 $ and
\begin{align*}
c_{k}=\left[C\frac{N(\gamma-1)+2 q_0(1-s)}{N(2^{k}-1)q_0} c_{k-1}^{\frac{(N(\gamma-1)+2^{k+1}q_0(1-s))}{N}}\right]^{\frac{N}{2(N(\gamma-1)+2^{k}q_0(1-s))}}, \vartheta_k=\frac{(1-2^{-k})N}{N(\gamma-1)+2q_0(1-s)}.
\end{align*}
Set $b_k=\log(c_k)$, we have
\begin{align*}
b_k&=\frac{N \log\left[C\frac{N(\gamma-1)+2 q_0(1-s)}{N(2^{k}-1)q_0}\right]}{2(N(\gamma-1)+2^{k}q_0(1-s))}+\frac{1}{2}\frac{(N(\gamma-1)+2^{k+1}q_0(1-s))}{(N(\gamma-1)+2^{k}q_0(1-s))} b_{k-1}.
\end{align*}
It follows,
\begin{align*}
|b_k|\leq  \frac{C}{(7/4)^k}+\frac{1}{2}\frac{(N(\gamma-1)+2^{k+1}q_0(1-s))}{(N(\gamma-1)+2^{k}q_0(1-s))} |b_{k-1}|.
\end{align*}
 There exists $k_0\geq 10$ such that
\begin{align*}
\frac{1}{2}\frac{(N(\gamma-1)+2^{k+1}q_0(1-s))}{(N(\gamma-1)+2^{k}q_0(1-s))}-\frac{4}{7}\geq  \frac{1}{7}~~\forall~~k\geq k_0.
\end{align*}
It is equivalent to
\begin{align*}
\frac{C}{(7/4)^k}\leq \frac{1}{2}\frac{(N(\gamma-1)+2^{k+1}q_0(1-s))}{(N(\gamma-1)+2^{k}q_0(1-s))} \frac{4C}{(7/4)^{k-1}}-\frac{4C}{(7/4)^{k}}~~\forall~~k\geq k_0.
\end{align*}
So,
\begin{align*}
|b_k|+\frac{4C}{(7/4)^{k}}\leq  \frac{1}{2}\frac{(N(\gamma-1)+2^{k+1}q_0(1-s))}{(N(\gamma-1)+2^{k}q_0(1-s))} \left[|b_{k-1}|+\frac{4C}{(7/4)^{k-1}}\right]~~~~\forall~~k\geq k_0.
\end{align*}
Thus,
\begin{align*}
|b_k|+\frac{4C}{(7/4)^{k}}\leq \frac{1}{2^{k-k_0}} \frac{(N(\gamma-1)+2^{k+1}q_0(1-s))}{(N(\gamma-1)+2^{k_0}q_0(1-s))} \left[|b_{k_0-1}|+\frac{4C}{(7/4)^{k_0-1}}\right]~~~~\forall~~k\geq k_0
\end{align*}
This means,
$$|b_k|\leq C~~\forall~~k\geq 0.$$
Hence, \eqref{es27} implies
\begin{align}
||u(t)||_{2^{k}q_0}=F_{2^{k}q_0}(t)^{-1}\leq  C t^{-\frac{(1-2^{-k})N}{N(\gamma-1)+2q_0(1-s)}}.
\end{align}
Using interpolation inequality, we get
\begin{align*}
||u(t)||_{q}\leq  C t^{-\frac{(1-\frac{q_0}{q})N}{N(\gamma-1)+2q_0(1-s)}}~~\forall ~q\geq q_0
\end{align*}
which implies \eqref{main-es5} for case $s>1-\frac{N}{2}$.\medskip\\
Similarly, we also obtain \eqref{main-es5} for case $s\leq 1-\frac{N}{2}$, we omit the details.\end{proof}
\begin{lemma}\label{le-6}	Assume $f=0$. let  $q_0\geq 1$ be such that  $N(\gamma-1)+2q_0(1-s)>0$. Then, \begin{align}\label{main-es6}
\int_{\Omega}| (-\Delta)^{\frac{1-s}{2}}(|u(t)|^{\frac{\gamma+q-1}{2}-1}u(t)|^2dx\leq  C ||u_0||_{L^{q_0}(\Omega)}^{\frac{N(\gamma-1)q_0+2q_0q(1-s)}{N(\gamma-1)+2q_0(1-s)}}t^{-\frac{(q-q_0)N}{N(\gamma-1)+2q_0(1-s)}-1}
	\end{align}
	for all $q\in [q_0,\infty)\cap(1,\infty).$
\end{lemma}
 \begin{proof} First, we prove that if $F:(0,\infty)\to (0,\infty)$ satisfies
 	\begin{align}\label{es63}
 	\int_{t/4}^{2t} F(s)ds\leq t^{-\alpha}~~\forall~~t>0,
 	\end{align}
 	for some $\alpha>0$, then
 	\begin{align}\label{es61}
 	F(t)\leq 2 t^{-\alpha-1}~~\forall~~t>0.
 	\end{align}
 	Indeed, let $\chi_\varepsilon$ be the standard mollifiers in $\mathbb{R}$ with $\supp \chi_\varepsilon\subset B_\varepsilon(0)$. Let $t>0$ be such that $
 	\lim\limits_{\varepsilon\to 0} (\chi_\varepsilon* F)(t)=F(t).$ We have for all $\varepsilon\in (0,t/8)$,
 	\begin{align*}
 	\int_{t/2}^{3t/2}(\chi_\varepsilon* F)(s)ds\leq \int_{t/4}^{2t} F(s)ds\leq t^{-\alpha}.
 	\end{align*}
 Applying a mean value principle to the smooth function $\chi_\varepsilon* F$  yields
 	\begin{align*}
 	(\chi_\varepsilon* F)(t)\leq \frac{2}{t}\max_{\tau\in [0,t/2]}\int_{t/2+\tau}^{t+\tau}(\chi_\varepsilon* F)(s)ds \leq  \frac{2}{t} \int_{t/2}^{3t/2}(\chi_\varepsilon* F)(s)ds
 	\leq  2t^{-\alpha-1}.
 	\end{align*}
 	Letting $\varepsilon\to 0$, we get \eqref{es61}.\\
 	By \eqref{main-es2}, we have \begin{align*}
 \int_{t/4}^{2t}\int_{\Omega}| (-\Delta)^{\frac{1-s}{2}}(|u|^{\frac{\gamma+q-1}{2}-1}u)|^2&\leq \int_\Omega |u(t/4)|^q\\&\overset{\eqref{main-es5}}\leq C ||u_0||_{L^{q_0}(\Omega)}^{\frac{N(\gamma-1)q_0+2q_0q(1-s)}{N(\gamma-1)+2q_0(1-s)}}t^{-\frac{(q-q_0)N}{N(\gamma-1)+2q_0(1-s)}}.
 	\end{align*}
 Applying \eqref{es63} to $\alpha=\frac{(q-q_0)N}{N(\gamma-1)+2q_0(1-s)}$ and $F(s)=\int_{\Omega}| (-\Delta)^{\frac{1-s}{2}}(|u(s)|^{\frac{\gamma+q-1}{2}-1}u(s)|^2dx$, we find \eqref{main-es6}.
 \end{proof}
\begin{lemma}\label{le-7} There exists a subsequence of ${u_\delta}$ converging to a solution $u$ of Problem \eqref{pro1}. Moreover, $u$ satisfies the properties stated in Lemmas \ref{le-1}, \ref{le-2}, \ref{le-3}, \ref{le-4},\ref{le-5} and \ref{le-6} with $\delta=0$.
\end{lemma}

\begin{proof}
From \eqref{main-es2'} and \eqref{main-es3}, we have
\begin{align}\label{es29}
\int_{\Omega_T}| (-\Delta)^{\frac{1-s}{2}}(|u_\delta|^{\frac{\gamma+p-1}{2}-1}u_\delta)|^2+||u_\delta||_{L^\infty(\Omega_T)}\leq C~~\forall~~p>1.
\end{align}
Set $
E_\delta:=\operatorname{div}(|u_\delta|^{m_1}\nabla (-\Delta)^{-s}(|u_\delta|^{m_2-1}u_\delta)).$
We prove that
\begin{align}\label{es28}
||E_\delta||_{L^2(0,T;(H^1_0(\Omega)\cap W^{2-\vartheta,r}(\Omega))^{*})}\leq C ~~ \text{for some } ,\vartheta\in (0,1),r\in (1,2).
\end{align}
Indeed, if $s\geq 1/2$, it is easy to find \eqref{es28} since
\begin{align*}
|||u_\delta|^{m_1}\nabla (-\Delta)^{-s}(|u_\delta|^{m_2-1}u_\delta)||_{L^2(\Omega_T)}\leq C.
\end{align*}
If $s<1/2$, we deduce from \eqref{es32*} in Lemma \eqref{le-f} below that
\begin{align*}
||E_\delta||_{L^2(0,T;(H^1_0(\Omega)\cap W^{2-\vartheta,r}(\Omega))^{*})}\leq |||u_\delta|^{\gamma-1}u_\delta||_{L^2(0,T,H^{1-2s}(\Omega))}\overset{\eqref{es29}}\leq C.
\end{align*}
It follows \eqref{es28}.
Hence, from \eqref{es28} and \eqref{es29} we have
\begin{align*}
||\partial_tu_\delta||_{L^2(0,T;(H^1_0(\Omega)\cap W^{2-\vartheta,r}(\Omega))^{*})}+\int_{\Omega_T}| (-\Delta)^{\frac{1-s}{2}}(|u_\delta|^{\frac{\gamma+p-1}{2}-1}u_\delta)|^2+||u_\delta||_{L^\infty(\Omega_T)}\leq C
\end{align*}
for some $r\in (1,2)$.
	By Lemma \ref{Simonlemma},, there exists a subsequence of $\{u_\delta\}$ converging to $u$ in $L^1(\Omega_T)$ as $\delta\to0$. Moreover, $u$ satisfies the properties stated in Lemmas \ref{le-1}, \ref{le-2}, \ref{le-3}, \ref{le-4},\ref{le-5} and \ref{le-6} with $\delta=0$. and
	\begin{align*}
	|u_\delta|^{\frac{\gamma+p-1}{2}-1}u_{\delta}\to |u|^{\frac{\gamma+p-1}{2}-1}u~~ L^2(0,T;H^{1-s-\varepsilon_0}(\Omega))~\forall~~\varepsilon_0>0, p>1.
	\end{align*}
From proof of Lemma \ref{le-f}, we see that
\begin{align*}
\operatorname{div}(|u_\delta|^{m_1}\nabla  (-\Delta)^{-s}  |u_\delta|^{m_2-1}u_\delta)\to \operatorname{div}(|u|^{m_1}\nabla  (-\Delta)^{-s}  |u|^{m_2-1}u)~~
\end{align*}
in $L^2(0,T;(H^1_0(\Omega)\cap W^{2-\vartheta,r}(\Omega))^{*})$ for some $\vartheta\in (0,1),r\in (1,\infty)$.
Thus, for $\varphi\in C_c^1([0,T),(W_0^{2,\infty}(\Omega))^*)$
\begin{align*}
\int_{\Omega_T}f\varphi+\int_{\Omega}u_0\varphi&=\int_{\Omega_T}(-\varphi_t-\delta\Delta \varphi)u_\delta+	\int_{\Omega_T}E_\delta\varphi \\&\to  \int_{\Omega}-\varphi_tu+	\int_0^{T}\langle\operatorname{div}(|u|^{m_1}\nabla (-\Delta)^{-s}(|u|^{m_2-1}u)),\varphi\rangle dt.
\end{align*}
which implies that $u$ satisfies
\begin{align*}
\int_{\Omega}-\varphi_tu dxdt+	\int_0^{T}\langle\operatorname{div}(|u|^{m_1}\nabla (-\Delta)^{-s}(|u|^{m_2-1}u)),\varphi\rangle dt=\int_{\Omega_T}f\varphi dxdt+\int_{\Omega}u_0\varphi dxdt,
\end{align*}
for all $\varphi\in C_c^1([0,T),(W_0^{2,\infty}(\Omega))^*)$. Hence, $u$
 is a solution of problem \eqref{pro1}. The proof is complete.
 \end{proof}


\section{Universal bound}
\label{sec.univbdd}

The property of universal boundedness depends on general arguments that we stress here because of their possible use in other settings. We recall that we consider equations with zero right-hand side.
Suppose we have already proved the a priori estimate
$$
\|u(t)\|_\infty\le C \|u_0\|_1^{\alpha} t^{-\beta},
$$
for some exponents $\alpha,\beta>0$, where a constant $C$ that does not depend on the data, it depends only on $N,s$ and $\Omega$. Suppose the equation has the following {\sl Invariance Property: If $u(x,t)$ is a solution in our admissible class,  so is}
\begin{equation}\label{transf}
u_k(x,t)= ku(x,k^{\gamma-1}t).
\end{equation}
In the case of our model \eqref{pro1} the result holds and $\gamma$ depends only on the powers of the equation, actually $\gamma=m_1+m_2$. We need to assume that $\gamma>1$.

\begin{proposition} Under those assumptions  we get the universal estimate
\begin{equation}\label{univ.est}
||u(t)||_{L^\infty(\Omega)}\le C_1(N,s,\gamma,\Omega) t^{-1/(\gamma-1)}\,
\end{equation}
valid for all solutions that we have constructed.
\end{proposition}

{\sl Proof.} (i) We begin with an initial data bounded above by constant $1$. Since $\Omega$ is bounded this datum is in $L^1(\Omega)$. We and use the a priori estimate to find a time $t_1=t_1$ such that
$$
||u(t)||_{L^\infty(\Omega)}\le \frac{1}{2(|\Omega|+1)} \quad \forall \ t\ge t_1.
$$

(ii)  Let us now apply the result to data with an estimate $||u_0||_{L^1(\Omega)}\le 2^j$. We define the new solution $u_k(x,t)=ku(x,k^{\gamma-1}t)$ with $k=2^{-j}$ and apply the previous step to show
that $$||u_k(t_1)||_{L^\infty(\Omega)}\le \frac{1}{2(|\Omega|+1)},$$ hence
$$
||u(t_j)||_{L^\infty(\Omega)}\le \frac1{2k(|\Omega|+1)}=\frac{2^{j-1}}{|\Omega|+1}\,,
$$
and
\begin{align*}
||u(t_j)||_{L^1(\Omega)}\le \frac{2^{j-1}|\Omega|}{|\Omega|+1}< 2^{j-1},
\end{align*}
when we put  $t_j=k^{\gamma-1} t_1=2^{-j(\gamma-1)}t_1$. We may now apply iteratively the argument after displacing the origin of time and get
\begin{align*}
||u(t_{j}+t_{j-1})||_{L^\infty(\Omega)}\le \frac{2^{j-2}}{|\Omega|+1},~~||u(t_{j}+t_{j-1})||_{L^1(\Omega)}\le 2^{j-2}
\end{align*}
so that
$$
||u(T_j)||_{L^\infty(\Omega)}\le \frac{2^{-1}}{|\Omega|+1} \leq 1\quad \mbox{for} \ T_j=\sum_{i=1}^{j}t_i=\sum_{i=1}^j  2^{-i(\gamma-1)}t_1=C(\gamma)t_1\,.
$$
the conclusion is that for data less than $2^j$ we need to wait $T_j$ seconds to get the bound $||u||_{L^\infty(\Omega)}\le 1$.

(iii) Consider now a general initial datum $u_0$, not necessarily integrable or bounded.
We approximate from below by bounded data and conclude that there is a limit solution
with the estimate
$$
||u(t)||_{L^\infty(\Omega)}\le 1 \quad \mbox{for some} \ t\leq  T_\infty:=\sum_{i=1}^\infty  2^{-j(\gamma-1)}t_1=\frac{t_1}{2^{\gamma-1}-1}.
$$
So,
$$
||u(t)||_{L^\infty(\Omega)}\le 1 \quad \mbox{for all} \ t\geq  T_\infty.
$$
 This estimate should be valid for all our constructed solutions. This is a particular case of the universal estimate.

(iv) To get  estimate \eqref{univ.est} for any $t=t_2>0$ fixed, use again the scaling $
u_k(x,t)=ku(x,k^{\gamma-1}t)$, now with $k^{\gamma-1}T_\infty=t_2 $ to get
$$
||u(t_2)||_{L^\infty(\Omega)}=(1/k)||u_k(T_\infty)||_{L^\infty(\Omega)}\le 1/k= (T_\infty /t_2)^{1/(\gamma-1)}\,.
$$
The estimate follows with $C_1=T_\infty^{1/(\gamma-1)}$. \hfill $\square$

\begin{remark}The result is not true for $m_1+m_2\le 1$ as many particular cases show.   Thus, when $m_1+m_2= 1$ any multiple of a solution is still a solution so that no a priori estimate may exist independent of the size of the initial data. For $m_1+m_2\leq1$ we have  the transformation \eqref{transf} but now with $\gamma-1\leq 0$. Suppose for contradiction that we have an a universal priori estimate
	$$
	||u(t)||_{L^\infty(\Omega)}\le C\,F(t)
	$$
	with $C$ a universal constant and $F(t)>0$ and nonincreasing. We consider $u_k(.,t)=ku(.,k^{\gamma-1}t)$. Then,
	$$
	k|| u(t)||_{L^\infty(\Omega)}= ||u_k(k^{1-\gamma}t)||_{L^\infty(\Omega)}\leq C F(k^{1-\gamma}t)\leq C F(t)~~\forall~~k\geq 1.
	$$
	Letting $k\to\infty$, we find the contradiction. We recall that sharp asymptotics in those cases have been explored for the fast diffusion equation $u_t-\Delta(|u|^{\gamma-1}u)=0$ and also the fractional porous medium $u_t+(-\Delta)^s (|u|^{\gamma-1}u)=0$. Phenomena of extinction in finite time occur.
\end{remark}


\section{Existence of solutions  with bad data}\label{sec.bad}

In this section, we establish the existence of solutions to Problem \eqref{pro1} with bad data.

\begin{theorem}[Distributional  data]  \label{pro-Dis-data}Let $f\in L^2(0,T;H^{-1+s}(\Omega))$ and  $u_0\in L^{\gamma+1}(\Omega)$. Then, Problem \eqref{pro1} admits a weak  solution $u\in C(0,T;L^{\gamma+1}(\Omega))$ satisfying
	\begin{align}
	\sup_{t\in (0,T)}\int_\Omega |u(t)|^{\gamma+1}+\int_{0}^{T}\int_{\Omega}| (-\Delta)^{\frac{1-s}{2}}(|u|^{\gamma})|^2
	\leq C \int_\Omega |u_0|^{\gamma+1}+ C\int_{0}^{T}\int_{\Omega} |(-\Delta)^{-\frac{1-s}{2}}f|^2\,.
	\end{align}
Moreover, when $f=0$ the Universal Bound \eqref{univ.bound} holds for these solutions.
\end{theorem}

\begin{proof}[Proof of Theorem \ref{pro-Dis-data}] Let $u_k$ be a solution of problem \eqref{pro1} in Theorem \ref{mainthm} with $u_{0}=T_k(u_0)$ and $f=f_k\in L^\infty(\Omega)$ such that $f_k\to f$ in $L^2(0,T;H^{-1+s}(\Omega))$ and
	$$\int_{0}^{t}\int_{\Omega} |(-\Delta)^{-\frac{1-s}{2}}f_k|^2\leq 2 \int_{0}^{t}\int_{\Omega} |(-\Delta)^{-\frac{1-s}{2}}f|^2.$$
We have  from \eqref{main-es2''"} of Lemma \ref{le-2} \begin{align*}
\sup_{t\in (0,T)}	\int_\Omega |u_k(t)|^{\gamma+1}+\int_{0}^{T}\int_{\Omega}| (-\Delta)^{\frac{1-s}{2}}(|u_k|^{\gamma-1}u_k)|^2
	&\leq C \int_\Omega |T_k(u_0)|^{\gamma+1}+ C\int_{0}^{t}\int_{\Omega} |(-\Delta)^{-\frac{1-s}{2}}f_k|^2\\&\leq C \int_\Omega |u_0|^{\gamma+1}+ C\int_{0}^{t}\int_{\Omega} |(-\Delta)^{-\frac{1-s}{2}}f|^2.
	\end{align*}
	So, by \eqref{es32*} in Lemma \ref{le-f}, we have
	\begin{align}
	||\operatorname{div}(|u_k|^{m_1}\nabla  (-\Delta)^{-s}  |u_k|^{m_2-1}u_k)||_{L^2(0,T;(H^1_0(\Omega)\cap W^{2-\vartheta,r}(\Omega))^{*})}\leq C ~~ \text{for some }\vartheta\in (0,1),r\in (1,\infty)
	\end{align}
	Thus,
	\begin{align*}
	||\partial_tu_k||_{L^{\min\{\gamma+1,2\}}(0,T;(H^1_0(\Omega)\cap W^{2-\vartheta,r}(\Omega))^{*})}+|||u_k|^{\gamma-1}u_k||_{L^2(0,T;H^{1-s}(\Omega))}+||u_k||_{L^\infty(0,T;L^{\gamma+1}(\Omega))}\leq C,
	\end{align*}
	for some $\vartheta\in (0,1),r\in (1,\infty)$.
	By Lemma \ref{Simonlemma}, there exists a subsequence of $\{u_\delta\}$ converging to $u$ in $L^1(
\Omega_T)$ as $\delta\to0$. Moreover, $u$ satisfies the properties stated in Lemmas \ref{le-1}, \ref{le-2}, \ref{le-3}, \ref{le-4},\ref{le-5} and \ref{le-6} with $\delta=0$. and the Universal bound \eqref{univ.est} and
	\begin{align*}
	|u_k|^{\gamma-1}u_k\to |u|^{\gamma-1}u~~\text{in}~~ L^2(0,T;H^{1-s-\varepsilon_0}(\Omega))~\forall~~\varepsilon_0>0.
	\end{align*}
	From proof of Lemma \ref{le-f}, we see that
	\begin{align*}
	\operatorname{div}(|u_k|^{m_1}\nabla  (-\Delta)^{-s}  |u_k|^{m_2-1}u_k)\to \operatorname{div}(|u|^{m_1}\nabla  (-\Delta)^{-s}  |u|^{m_2-1}u),~~
	\end{align*}
	in $L^1(0,T;W^{-2,r}(\Omega))$ for some $r\in (1,2)$. It is easy to check that $u$ is a weak solution of problem \eqref{pro1}. The proof is complete.
\end{proof}

We need a new definition of solution when the data are measures.

	\begin{definition} Let $\mu\in \mathcal{M}_b(\Omega_T),\sigma\in \mathcal{M}_b(\Omega)$. We say that $u$ is a distribution solution  of problem \eqref{pro1} with $(f,u_0)=(\mu,\sigma)$, if\\
		(i) $u\in L^{1}(\Omega_T)$,\\ (ii) $\chi\operatorname{div}(|u|^{m_1}\nabla (-\Delta)^{-s}|u|^{m_2-1}u)\in L^1(0,T,(W_0^{2,\infty}(\Omega))^*)$ for any $\chi\in C^\infty_c(\Omega\times[0,T)$ and
		\begin{align*}
		-\int_{0}^{T}\int_\Omega u\phi_tdxdt-\int_{0}^{T} \langle \operatorname{div}(|u|^{m_1}\nabla (-\Delta)^{-s}|u|^{m_2-1}u),\varphi\rangle dt=\int_\Omega\phi(0)d\sigma+\int_{0}^{T}\int_\Omega d\mu
		\end{align*}
		for all $\phi\in C_c^2(\Omega\times[0,T))$.
	\end{definition}
Here and in what follows, we denote by
$\mathcal{M}_b(D)$,  the set of bounded Radon measures in a set $D$.
We can state the following  theorem.

	\begin{theorem}[Measure data] \label{pro-measuredata}Let $\mu\in \mathcal{M}_b(\Omega_T)$ and $\sigma\in\mathcal{M}_b(\Omega)$. Assume that  $\gamma>1-\frac{2(1-s)}{N}$.Then, the Problem \eqref{pro1} admits a distribution solution satisfying
	\begin{align}
&1_{s<1-\frac{N}{2}}||u||_{L^{\gamma+\frac{2(1-s)}{N},\infty}(\Omega_T)}+1_{s=1-\frac{N}{2}}||u||_{L^{\gamma+1-\frac{1}{l},\infty}(\Omega_T)}+1_{s>1-\frac{N}{2}}
||u||_{L^{\gamma+1,\infty}(\Omega_T)}\\&\nonumber\leq C 1_{s<1-\frac{N}{2}}M^{\frac{N+2(1-s)}{\gamma N-2(1-s)}}+C1_{s=1-\frac{N}{2}}M^{\frac{2l}{l(\gamma+1)-1}}+C1_{s>1-\frac{N}{2}}M^{\frac{2}{\gamma+1}},
\end{align}
for all $l>1$ and
\begin{align}
\int_{0}^T\int_\Omega|(-\Delta)^{\frac{1-s}{2}}\left(\frac{|u|^{\frac{\gamma}{2}+\theta-1}u}{|u|^{2\theta}+1}\right)|^2dxdt\leq C_\theta M~~ \forall~~\theta>0
\end{align}
where $M=||u_0||_{\mathcal{M}_b(\Omega)}+||f||_{\mathcal{M}_b(\Omega_T)}$.		
Moreover, the  Smoothing Effect \eqref{sm.eff} and the Universal Bound \eqref{univ.bound} holds for these solutions.
	\end{theorem}

\begin{proof}[Proof of Theorem \ref{pro-measuredata}]  Let $\sigma_n,\mu_n$ be in  $L^\infty(\Omega)$ and $L^\infty(\Omega_T)$ converging weakly to $\sigma$ and $\mu$ in $\mathcal{M}_b(\Omega)$ and $\mathcal{M}_b(\Omega_T)$ such that
	\begin{align}
	|\sigma_n|(\Omega)\leq |\sigma|(\Omega),~~|\mu_n|(\Omega_T)\leq |\mu|(\Omega_T)~~\forall~~k\in\mathbb{N}.
	\end{align}
	Let $u_n$ be a solution of problem \eqref{pro1} in Theorem \ref{mainthm} with $u_0=\sigma_n$ and $f=\mu_n$.
	We have
	\begin{align} \label{es50}
	&1_{s<1-\frac{N}{2}}||u_n||_{L^{\gamma+\frac{2(1-s)}{N},\infty}(\Omega_T)}+1_{s=1-\frac{N}{2}}||u_n||_{L^{\gamma+1-\frac{1}{r},\infty}(\Omega_T)}+1_{s>1-\frac{N}{2}}||u_n||_{L^{\gamma+1,\infty}(\Omega_T)}\leq C
	\end{align}
	and
	\begin{align}
	|||u_n||_{L^\infty(0,T,L^1(\Omega))}\leq  C,
	\end{align}
	\begin{align}\label{es47}
	\int_{0}^T\int_\Omega|(-\Delta)^{\frac{1-s}{2}}\left(\frac{|u_n|^{\frac{\gamma}{2}+\theta-1}u_n}{|u_n|^{2\theta}+1}\right)|^2dxdt\leq C~~ \forall~~\theta>0
	\end{align}
	Since $\gamma>1-2(1-s)/N$, then we obtain from \eqref{es50}
	\begin{align}\label{es51}
	||u_n||_{L^{q}(\Omega_T)}\leq C ~~\text{for some }~~q>\max\{\gamma,1\}.
	\end{align}
	Now, we prove that for that for any $\epsilon\in (0,(1-s)/2)$,
	\begin{align}\label{es47'}
	\int_{0}^T\int_\Omega|(-\Delta)^{\frac{1-s}{2}-\epsilon} \left(|u_n|^{\frac{\gamma}{2}+\epsilon_0-1}u_n\right)|^2dxdt\leq C,
	\end{align}
	for some $\epsilon_0>0$. In particular, for any $\varrho>0$
	\begin{align}\label{es56}
	|||u_n|^{\frac{\gamma}{2}-1}u_n||_{L^2(0,T;H^{1-s-\varrho}(\Omega))}\leq C~~\forall ~n.
	\end{align}
	Indeed, it is not hard to check that for any $\theta_0\in (0,\gamma/10)$
	\begin{align*}
	\left|\frac{|x_1|^{\frac{\gamma}{2}+\theta-1}x_1}{|x_1|^{2\theta}+1}-\frac{|x_2|^{\frac{\gamma}{2}+\theta-1}x_2}{|x_2|^{2\theta}+1}\right|\geq C(\theta,\theta_0,\gamma) \frac{||x_1|^{\frac{\gamma}{2}-\theta_0-1}x_1-|x_2|^{\frac{\gamma}{2}-\theta_0-1}x_2|^{\frac{\gamma+2\theta}{\gamma-2\theta_0}}}{||x_1|^{\frac{\gamma}{2}-\theta_0-1}x_1-|x_2|^{\frac{\gamma}{2}-\theta_0-1}x_2|^{\frac{4\theta}{\gamma-2\theta_0}}+1}~\forall x_1,x_2\in \mathbb{R},
	\end{align*}
	for all $\theta<<\theta_0$. We obtain from  \eqref{es47}, \eqref{es51}, and \eqref{L^p-b-frac1} that for any $\theta_0\in (0,\gamma/10)$
	\begin{align*}
&\int_{0}^{T}\int_{\Omega}\int_\Omega |u(x,t)|^{\gamma\nu}+|u(y,t)|^{\gamma\nu}+1\\&+\frac{1}{|x-y|^{N+2(1-s)}}\left(\frac{||u_n(x,t)|^{\frac{\gamma}{2}-\theta_0-1}u_n(x,t)-|u_n(y,t)|^{\frac{\gamma}{2}-\theta_0-1}u_n(y,t)|^{\frac{\gamma+2\theta}{\gamma-2\theta_0}}}{||u_n(x,t)|^{\frac{\gamma}{2}-\theta_0-1}u_n(x,t)-|u_n(y,t)|^{\frac{\gamma}{2}-\theta_0-1}u_n(y,t)|^{\frac{4\theta}{\gamma-2\theta_0}}+1}\right)^2 dxdydt\leq C,
	\end{align*}
	for some $\nu>1$.
	Using H\"older's inequality for $(\frac{\gamma+2\theta}{2(\theta+\theta_0)},\frac{\gamma+2\theta}{\gamma-2\theta_0})$,
		\begin{align*}
	&\int_{0}^{T}\int_{\Omega}\int_\Omega\frac{1}{|x-y|^{\frac{\gamma-2\theta_0}{\gamma+2\theta}(N+2(1-s))}} E_{\theta,\theta_0}(x,y)^2 dxdydt\leq C,
	\end{align*}
	where
	\begin{align*}
	E_{\theta,\theta_0}(x,y)&=\frac{||u_n(x,t)|^{\frac{\gamma}{2}-\theta_0-1}u_n(x,t)-|u_n(y,t)|^{\frac{\gamma}{2}-\theta_0-1}u_n(y,t)|}{||u_n(x,t)|^{\frac{\gamma}{2}-\theta_0}+|u_n(y,t)|^{\frac{\gamma}{2}-\theta_0}|^{\frac{4\theta}{\gamma+2\theta}}+1}\left(|u(x,t)|^{\gamma\nu}+|u(y,t)|^{\gamma\nu}+1\right)^{\frac{\theta+\theta_0}{\gamma+2\theta}}.
	\end{align*}
Note that for $\theta>0$ small
	\begin{align*}
	E_{\theta,\theta_0}(x,y)&\geq C_\theta ||u_n(x,t)|^{\frac{\gamma}{2}-\theta_0-1}u_n(x,t)-|u_n(y,t)|^{\frac{\gamma}{2}-\theta_0-1}u_n(y,t)|(|u(x,t)|+|u_n(y,t)|)^{\nu\theta_0-\circ(1)}
\\	&\geq C_\theta ||u_n(x,t)|^{\frac{\gamma}{2}+(\nu-1)\theta_0-1-c_\theta}u_n(x,t)-|u_n(y,t)|^{\frac{\gamma}{2}+(\nu-1)\theta_0-1-c_\theta}u_n(y,t)|,
	\end{align*}
 where $c_\theta>0$ and $c_\theta\to 0$ as $\theta\to0$. Thus we deduce \eqref{es47'}.
	Let $\beta=\frac{s+2}{4}$. Let $\chi_k$ be smooth function in $\Omega$ such that $\chi_k=1$ in $\Omega\backslash\Omega_{1/k}$ and $\chi_k=0$ in $\Omega_{1/2k}$.
	Now, we  prove that
	\begin{align}\label{es53}
	||\chi_k\operatorname{div}(|u_n|^{\gamma/2}\nabla  (-\Delta)^{-s}  |u_n|^{\gamma/2-1}u_n)||_{L^1(0,T;(H^1_0(\Omega)\cap W^{2,\nu}(\Omega))^*)}\leq C_k~~\forall~~n\in\mathbb{N}.
	\end{align}
	for some $\nu>2$.
We have for $\varphi\in L^\infty(0,T,W^{1,\infty}(\Omega))$,
		\begin{align}\label{es52}
	&	I=\left|\int_{0}^{T}\int_\Omega\chi_k\operatorname{div}(|u_n|^{m_1}\nabla  (-\Delta)^{-s}  |u_n|^{m_2-1}u_n)\varphi dxdt\right|\\&\nonumber\leq \left|\int_{0}^{T}\int_\Omega|u_n|^{m_1}\nabla  (-\Delta)^{-s}  (\chi_{4k}|u_n|^{m_2-1}u_n)\nabla (\chi_k\varphi) dxdt\right|\\&\nonumber+\left|\int_{0}^{T}\int_\Omega|u_n|^{m_1}\nabla  (-\Delta)^{-s}  ((1-\chi_{4k})|u_n|^{m_2-1}u_n)\nabla (\chi_k\varphi) dxdt\right|\\&:= I_1+I_2\nonumber.
	\end{align}
	\textbf{Estimate}: $I_2$
	\begin{align*}
	I_2&\leq \int_{0}^{T}\int_\Omega|u_n|^{m_1}dx||\nabla  (-\Delta)^{-s}  ((1-\chi_{4k})|u_n|^{m_2-1}u_n)||_{L^\infty(\Omega\backslash\Omega_{k})}||\nabla (\chi_k\varphi)||_{L^\infty(\Omega)} dt
	\\&\leq C_k \int_{0}^{T}||u_n||_{L^{m_1}(\Omega)}^{m_1}||u_n||_{L^{m_2}(\Omega)}^{m_2}||\varphi||_{W^{1,\infty}(\Omega)} dt\\& \overset{\eqref{es51}} \leq C_k ||\varphi||_{L^{\nu}(0,T;W^{1,\infty}(\Omega))}
	\end{align*}
	for some $\nu>2$.

\textbf{Estimate}: $I_1$. It is easy to see that if $s\geq 1/2$,
		\begin{align*}
		I_1\leq C_k ||\varphi||_{L^{\nu}(0,T;W^{1,\infty}(\Omega))}
		\end{align*}
			for some $\nu>2$. 		So, it is enough to assume $s\in (0,1/2)$. We will prove that
			\begin{align}\label{es57}
		I_1\leq C_k ||\varphi||_{L^{\nu_1}(0,T;W^{2,\nu_2}(\Omega))}
		\end{align} holds
		for some $\nu_1,\nu_2>1.$ Indeed, for $\beta\in (0,1/2)$
		\begin{align*}
		I_1&=\left|\int_{0}^{T}\int_\Omega  (-\Delta)^{\beta-s}  (\chi_{4k}|u_n|^{m_2-1}u_n) (-\Delta)^{-\beta}\operatorname{div}\left(|u_n|^{m_1}\nabla (\chi_k\varphi)\right) dxdt\right|\\&\leq
		\int_{0}^{T} ||(-\Delta)^{\beta-s}  (\chi_{4k}|u_n|^{m_2-1}u_n)||_{L^{\frac{\gamma}{m_2}}(\Omega)}||(-\Delta)^{-\beta}\operatorname{div}\left(|u_n|^{m_1}\nabla (\chi_k\varphi)\right)||_{L^{\frac{\gamma}{m_1}}(\Omega)} dt
		\end{align*}
By Lemma \ref{fractional-lem}, 		
	\begin{align*}
I_1&\leq
C_k\int_{0}^{T} ||\chi_{4k}|u_n|^{m_2-1}u_n||_{W^{2(\beta-s)^+,\frac{\gamma}{m_2}}(\Omega)}|||u_n|^{m_1}\nabla (\chi_k\varphi)||_{W^{1-2\beta,\frac{\gamma}{m_1}}(\Omega)} dt\\&\leq
C_k\int_{0}^{T} ||\chi_{4k}|u_n|^{m_2-1}u_n||_{W^{2(\beta-s)^+,\frac{\gamma}{m_2}}(\Omega)}||\chi_{4k}|u_n|^{m_1}||_{W^{1-2\beta,\frac{\gamma}{m_1}}(\Omega)}||\varphi||_{W^{2,2N/\beta}(\Omega)} dt.
\end{align*}

\noindent\textbf{Case 1.}  $m_1=m_2$. We take $\beta=\frac{1+2s}{4}$. Using interpolation inequality yields
	 \begin{align*}
	 I_1&\leq
	 C_k\int_{0}^{T} |||u_n|^{\frac{\gamma}{2}-1}u_n||_{H^{1-2s}(\Omega)}^2||\varphi||_{W^{2,\infty}(\Omega)} dt\\&\leq
	 C_k\int_{0}^{T} |||u_n|^{\frac{\gamma}{2}-1}u_n||_{H^{1-3s/2}(\Omega)}^{\frac{4(1-2s)}{2-3s}}|||u_n|^{\frac{\gamma}{2}}||_{L^2(\Omega)}^{\frac{2s}{2-3s}}||\varphi||_{W^{2,2N/\beta}(\Omega)} dt\\&\overset{\eqref{es51}}  \leq C_k\int_{0}^{T} |||u_n|^{\frac{\gamma}{2}-1}u_n||_{H^{1-3s/2}(\Omega)}^{\frac{4(1-2s)}{2-3s}}
||\varphi||_{W^{2,2N/\beta}(\Omega)} dt
\\&\overset{\eqref{es56}}\leq C_k||\varphi||_{L^\nu(0,T,W^{2,2N/\beta}(\Omega))}
	 \end{align*}
	 for some $\nu>2$. So, we get \eqref{es57}.

\noindent	 \textbf{Case 2.} $m_1<m_2$.
	 We take $\beta\in (s,1/2)$ such that $
	 \frac{2m_1}{\gamma}(1-s)>1-2\beta.$\\
	 As \eqref{es76} we have
	 \begin{align}\label{es58}
	 ||u_n|^{\frac{\gamma}{2}-1}u_n||_{H^{1-s-\varrho}(\Omega)}^{\frac{2m_1}{\gamma}}&\geq C |||u_n|^{m_1-1}u_n||_{W^{\frac{2m_1(1-s-\varrho)}{\gamma},\frac{\gamma}{m_1}}(\Omega)}\\&\geq C |||u_n|^{m_1-1}u_n||_{W^{1-2\beta,\frac{\gamma}{m_1}}(\Omega)}\nonumber.
	 \end{align}
	 for $\varrho>0$ small enough. 	 By \cite[Proposition 5.1, Chapter 2]{Taylor}
	 \begin{align}
	 \label{es60}&||\chi_{4k}|u_n|^{m_2-1}u_n||_{W^{2(\beta-s),\frac{\gamma}{m_2}}(\Omega)}\\&~~\leq C||(\chi_{4k}^{\frac{\gamma}{2m_2}}|u_n|^{\frac{\gamma}{2}})^{\frac{m_2-m_1}{\gamma}}||_{L^{p_1}(\mathbb{R}^N)} ||\chi_{4k}^{\frac{\gamma}{2m_2}}|u_n|^{\frac{\gamma}{2}-1}u_n||_{W^{2(\beta-s),p_2}(\mathbb{R}^N)}\nonumber
	 \\&~~\leq \nonumber C|||u_n|^{\frac{\gamma}{2}}||_{L^{\frac{p_1(m_2-m_1)}{\gamma}}(\Omega)}^{\frac{m_2-m_1}{\gamma}} |||u_n|^{\frac{\gamma}{2}-1}u_n||_{H^{2(\beta-s)}(\Omega)}
	 \\&~~\leq C|||u_n|^{\frac{\gamma}{2}-1}u_n||_{H^{1-s-\varrho}}^{\frac{2m_2}{\gamma}} \nonumber
	 \end{align}
	 for $\varrho>0$ small enough, where
	 $p_1=\frac{2N\gamma}{(m_2-m_1)(N-2(1-s-\varrho))}$,  $p_2=\frac{2N\gamma}{\gamma N+2(1-s-\varrho)(m_2-m_1)}$ if $N-2(1-s)\geq 0$, and $p_1=1, p_2=\frac{\gamma}{m_2}$ if $N-2(1-s)< 0.$

	 So, it follows from \eqref{es58} and \eqref{es60}   that
	 	\begin{align*}
	 I_1&\leq
	 C_k\int_{0}^{T} |||u_n|^{\frac{\gamma}{2}-1}u_n||_{H^{1-s-\varrho}}^2||\varphi||_{W^{2,2N/\beta}(\Omega)} dt.
	 \end{align*}
	 Thus, as in \textbf{case 1.}, using interpolation inequality and \eqref{es51} we get
	 \begin{align*}
	 I_1&\leq
	 C_k ||\varphi||_{L^{\nu}(0,T,W^{2,2N/\beta}(\Omega))}.
	 \end{align*}
	 for some $\nu>1$.

\noindent	  \textbf{Case 3.} $m_2>m_1$. Similarly, we also obtain
	  \begin{align*}
	  I_1&\leq
	  C_k ||\varphi||_{L^{\nu}(0,T,W^{2,2N/\beta}(\Omega))}.
	  \end{align*}
	  for some $\nu>1$. Therefore, we deduce \eqref{es53}.

Since
\begin{align*}
(\chi_k u_n)_t = \chi_k f_n+\chi_k\operatorname{div}(|u_n|^{m_1}\nabla  (-\Delta)^{-s}  |u_n|^{m_2-1}u_n),
\end{align*}
thus for any $\varrho>0$, $k\geq 1$
	\begin{align*}
	||\partial_t (\chi_ku_n)||_{L^1(0,T;(W^{2,\nu}(\Omega))^*)}+||u_n||_{L^{q}(\Omega_T)}+|||u_n|^{\frac{\gamma}{2}}||_{L^2(0,T,H^{1-s-\varrho}(\Omega))}\leq C_k ~~\forall~~n,
	\end{align*}
	for some $\nu\geq 2$.
	By Lemma \ref{Simonlemma}, there exists a subsequence of $\{u_n\}$ converging to $u$ in $L^1(\Omega_T)$ as $n\to\infty$. Moreover,  $u$ satisfies the properties stated in Lemmas \ref{le-1}, \ref{le-2}, \ref{le-3}, \ref{le-4},\ref{le-5} and \ref{le-6} with $\delta=0$ and the Universal bound \eqref{univ.est}   and
	\begin{align*}
	|u_n|^{\frac{\gamma}{2}-1}u_n\to |u|^{\frac{\gamma}{2}-1}u~~\text{in}~~ L^2(0,T;H^{1-s-\varepsilon_0}(\Omega))~\forall~~\varepsilon_0>0
	\end{align*}
	Thus, we derive from proof of \eqref{es53} that	\begin{align*}
	&	|\int_{\Omega_T}\operatorname{div}(|u_n|^{m_1}\nabla  (-\Delta)^{-s}  |u_n|^{m_2-1}u_n)\varphi dx-\int_{\Omega_T}\operatorname{div}(|u|^{m_1}\nabla  (-\Delta)^{-s}  |u|^{m_2-1}u)\varphi dx|\\&\leq
		 \left|\int_{0}^{T}\int_\Omega (-\Delta)^{\beta-s}  (|u_n|^{\gamma/2-1}u_n- |u|^{\gamma/2-1}u) (-\Delta)^{-\beta}\left[\operatorname{div}\left(|u_n|^{m_1}\nabla \varphi\right)\right] dxdt\right|
	\\&+ \left|\int_{0}^{T}\int_\Omega (-\Delta)^{\beta-s}  (|u|^{m_2-1}u) (-\Delta)^{-\beta}\left[\operatorname{div}\left(\left(|u_n|^{m_1}-|u|^{m_1}\right)\nabla \varphi\right)\right] dxdt\right|~~\text{as}~~n\to\infty.
	\end{align*}
	for any $\varphi\in C^2_c(\Omega\times[0,T))$.  Thus, $u$ is a distribution solution of problem \eqref{pro1}. The proof is complete.
\end{proof}

\section{Appendix}\label{sec.prelim}

In this section, we collect some basic estimates of the semi-group $e^{t\Delta}$ and the fractional operator that we have used throughout the paper.

\begin{lemma} \label{le12}Let $e^{t\Delta}$ be the semi-group in bounded domain $\Omega$. Then, the following properties hold
	\begin{align}\label{para3}
	||e^{t\Delta}u_0||_{L^\infty(\Omega)}\leq C||u_0||_{L^\infty(\Omega)},
	\end{align}
	\begin{align}\label{para3+}
	||e^{t\Delta}\operatorname{div}(g)||_{L^\infty(\Omega)}\leq \frac{C}{\sqrt{t}} ||g||_{L^\infty(\Omega)},
	\end{align}
	and
		\begin{align}\label{para5} ||\nabla e^{t\Delta}u_0||_{L^\infty(\Omega)}\leq \frac{C}{\sqrt{t}} ||u_0||_{L^\infty(\Omega)}.
	\end{align}
\end{lemma}
We refer to \cite{QH1,QH2} for $L^p$ estimates for the semi-group $e^{t\Delta}$.
 Let $\mathbf{H}(t,x,y)$  be the Heat kernel in $\Omega\times (0,\infty)$. We recall some basic properties of $\mathbf{H}$, 	see \cite{Con-Igna},
 \begin{align}\label{heatkernel1}
 H(t,x,y)\leq C\min\left\{\frac{d(x)}{|x-y|},1\right\}\min\left\{\frac{d(y)}{|x-y|},1\right\}t^{-\frac{N}{2}}\exp(-c\frac{|x-y|^2}{t})
 \end{align}
 and
 \begin{align}\label{heatkernel2}
 |\nabla_x H(t,x,y)|\leq C \left[\frac{1}{d(x)}1_{\sqrt{t}\geq d(x)}+(\frac{1}{\sqrt{t}}+\frac{|x-y|}{t})1_{\sqrt{t}<d(x)}\right]H(t,x,y)
 \end{align}
 \begin{align}\label{heatkernel3}
 |\nabla_y H(t,x,y)|\leq C \left[\frac{1}{d(y)}1_{\sqrt{t}\geq d(y)}+(\frac{1}{\sqrt{t}}+\frac{|x-y|}{t})1_{\sqrt{t}<d(y)}\right]H(t,x,y)\,.
 \end{align}
It is not hard to show that these properties of the Heat kernel imply \eqref{para3},\eqref{para3+} and \eqref{para5}. We omit the details.
\begin{lemma} \label{Lpes}Let $p>1$ and $\beta\in (0,1]$. Then,
	if  $\beta< N/2p $
	\begin{align}\label{L^p-b-frac1}
	||f||_{L^{\frac{pN}{N-2p\beta}(\Omega)}}\leq ||(-\Delta)^{\beta}f||_{L^p(\Omega)}~~\forall~~f\in L^p(\Omega),
	\end{align}
	if $\beta=\frac{N}{2p}$
	\begin{align}\label{L^p-b-frac2}
	||f||_{BMO(\Omega)}\leq ||(-\Delta)^{\beta}f||_{L^p(\Omega)}~~\forall~~f\in L^p(\Omega),
	\end{align}
	if $\beta>\frac{N}{2p}$
	\begin{align}\label{L^p-b-frac3}
	||f||_{L^\infty(\Omega)}\leq ||(-\Delta)^{\beta}f||_{L^p(\Omega)}~~\forall~~f\in L^p(\Omega).
	\end{align}
\end{lemma}
\begin{proof} By \cite[Theorem 2.4]{ca1}, we have
\begin{align*}
|(-\Delta)^{-\beta}f(x)|\leq C \int_{0}^{R}\frac{\int_{B_\rho(x)}|f(y)|dy}{\rho^{N-2\beta}}\frac{d\rho}{\rho}~~\forall~~x\in \Omega,
\end{align*}
with $R=2\,diam(\Omega)$. Thus, by the standard potential estimate, see \cite{Adma} we get \eqref{L^p-b-frac1}, \eqref{L^p-b-frac2} and \eqref{L^p-b-frac3}.
\end{proof}
\begin{lemma}\label{le14} Let $\beta\in (\frac{1}{2},1]$. Then,
		\begin{align}\label{es2}
	||\nabla (-\Delta)^{-\beta}f||_{L^\infty(\Omega)}\leq C||f||_{L^\infty(\Omega)}
	\end{align}
\end{lemma}
	\begin{proof} Its proof can be found in \cite[Theorem 5]{ca1}.
\end{proof}
\begin{lemma}\label{le-H} Let $\beta\in (0,\frac{1}{2}]$. Then,
	\begin{align}\label{es49}
	||(-\Delta)^{-\beta}\dive(g)||_{L^2(\Omega)}\leq C||g||_{H^{1-2\beta}(\Omega)}~~\forall~~g\in H^{1-2\beta}(\Omega).
	\end{align}
\end{lemma}
\begin{proof}Set $v=(-\Delta)^{-1}\operatorname{div}(g)$, by the standard regularity theorem for Laplace we have
	\begin{align*}
	||v||_{H^{\varrho+1}(\Omega)}\leq C ||g||_{H^{\varrho}(\Omega)}~~\forall~~\varrho\in [0,1]
	\end{align*}
	see \cite{JerKenig}, it follows
	\begin{align*}
	||(-\Delta)^{-\beta}\operatorname{div}(g)||_{L^2(\Omega)}=||(-\Delta)^{1-\beta}v||_{L^2(\Omega)}\leq C ||v||_{H^{2(1-\beta)}(\Omega)}
	\leq C ||g||_{H^{1-2\beta}(\Omega)},
	\end{align*}
	so, we find \eqref{es49}. The proof is complete.\\
\end{proof}
\begin{lemma}\label{fractional-lem} Let $\beta\in (0,1/2),p>1$ and $\varepsilon>0$. Then,
	\begin{align}\label{es55}
	||(-\Delta)^{\beta}h||_{L^p(\Omega)}\leq C||h||_{W^{2\beta,p}(\Omega)}
	\end{align}
	for all $h\in W^{2\beta,p}(\Omega)$ $\supp h\subset \Omega\backslash\Omega_{\varepsilon_0}$ and
	\begin{align}\label{es54}
	||(-\Delta)^{-\beta}\operatorname{div}(g)||_{L^p(\Omega)}\leq C||g||_{W^{1-2\beta,p}(\Omega)}
	\end{align}
	for all $g\in W^{1-2\beta,p}(\Omega)$ $\supp g\subset \Omega\backslash\Omega_{\varepsilon_0}$.	
\end{lemma}

\begin{proof}
	{\bf 1.} We have
	\begin{align*}
	(-\Delta)^{\beta}h(x)=C_\beta\int_{0}^{\infty}t^{-1-\beta}(h(x)-v(x,t))dt
	\end{align*}
	where $v$ is a solution of problem
	\begin{align*}
	\left\{
	\begin{array}
	[c]{l}%
	\partial_tv-\Delta v=0~~\text{in }\Omega\times (0,\infty),\\
	v=0~~~~~~~~~~~~\text{on}
	~~\partial\Omega\times (0,\infty) 	\\
	v(0)=h~~~~~~~~~~~~~~~~\text{in}
	~~\Omega 	\\
	\end{array}
	\right.
	\end{align*}

	Since $\supp(h)\subset \Omega\backslash\Omega_{\varepsilon_0}$,
	\begin{align*}
	||v||_{C^{1,2}(\overline{\Omega}_{\varepsilon_0/2}\times [0,T])}\leq C||h||_{L^1(\Omega)}.
	\end{align*}
	Let $\chi\in C_c^\infty(\Omega)$ be such that $\chi=1$  in $\Omega\backslash\Omega_{\varepsilon_0/4} $ and $\chi=0$ in $\Omega_{\varepsilon_0/8}$. We have
	\begin{align*}
	\left\{
	\begin{array}
	[c]{l}%
	\partial_t(\chi v)-\Delta (\chi v)=-\Delta\chi v-2\nabla\chi\nabla v~~\text{in }\mathbb{R}^N\times (0,\infty),\\
	\chi v(0)=\chi h~~~~~~~~~~~~~~~~\text{in}
	~~\mathbb{R}^N
	\\
	\end{array}
	\right.
	\end{align*}
	Set $V=e^{t\Delta_{\mathbb{R}^N}}(\chi h)-\chi(x) v(t,x)$. Then,
	\begin{align*}
	\left\{
	\begin{array}
	[c]{l}%
	\partial_tV-\Delta V=-\Delta\chi v-2\nabla\chi\nabla u~~\text{in }\mathbb{R}^N\times (0,\infty),\\
	V(0)=0~~~~~~~~~~~~~~~~\text{in}
	~~\mathbb{R}^N
	\\
	\end{array}
	\right.
	\end{align*}
	Clearly, $
	||V||_{C^{2,1}(\mathbb{R}^N\times(0,\infty))}\leq 	C||v||_{C^{1,2}(\overline{\Omega}_{\varepsilon_0/2}}\leq C||h||_{L^1(\Omega)}.$
	So,
$$	|V(t,x)| \leq  C\min\{t,t^{-4}\}||g||_{L^1(\Omega)}\,,
$$
and
	\begin{align*}
	&||(-\Delta)^{\beta}h||_{L^p(\Omega)}=	||(-\Delta)_{\mathbb{R}^N}^{\beta}(\chi h)+C_\beta\int_{0}^{\infty}t^{-1-\beta}V(.,t)dt||_{L^p(\Omega)}\\&\leq C ||(-\Delta)_{\mathbb{R}^N}^{\beta}(\chi h)||_{L^p(\mathbb{R}^N)}+C||h||_{L^1(\Omega)}
	\leq C||h||_{W^{2\beta,p}(\Omega)}\,,
	\end{align*}
	which implies \eqref{es55}.

	{\bf 2.} We have
	\begin{align*}
	(-\Delta)^{-\beta}\operatorname{div}(g)(x)=C_\beta\int_{0}^{\infty}t^{-1+\beta}w(x,t)dt
	\end{align*}
	where $v$ is a solution of
	\begin{align*}
	\left\{
	\begin{array}
	[c]{l}%
	\partial_tw-\Delta w=0~~\text{in }\Omega\times (0,\infty),\\
	w=0~~~~~~~~~~~~\text{on}
	~~\partial\Omega\times (0,\infty)
	\\
	w(0)=\operatorname{div}(g)~~~~~~~~~~~~~~~~\text{in}
	~~\Omega
	\\
	\end{array}
	\right.
	\end{align*}
	Since $\supp(g)\subset \Omega\backslash\Omega_{\varepsilon_0}$,
	\begin{align*}
	||v||_{C^{1,2}(\overline{\Omega}_{\varepsilon_0/2}\times [0,T])}\leq C||g||_{L^1(\Omega)}.
	\end{align*}
	As above, we get
	\begin{align*}
	|e^{t\Delta_{\mathbb{R}^N}}\operatorname{div}(\chi g)(x)-\chi (x)w(x,t)|\leq C\min\{t,t^{-4}\}||g||_{L^1(\Omega)}\,,
	\end{align*}
thus,
	\begin{align*}
		||(-\Delta)^{-\beta}\operatorname{div}(g)||_{L^p(\Omega)}&=	||(-\Delta)_{\mathbb{R}^N}^{-\beta}\operatorname{div}(\chi g)-C_\beta\int_{0}^{\infty}t^{-1+\beta}e^{t\Delta_{\mathbb{R}^N}}\operatorname{div}(\chi g)(x)-\chi (x)w(x,t)dt||_{L^p(\Omega)}\\&\leq C ||(-\Delta)_{\mathbb{R}^N}^{-\beta}\operatorname{div}(\chi g)||_{L^p(\mathbb{R}^N)}+C||g||_{L^1(\Omega)}
		\\&\leq C ||\operatorname{div}\left((-\Delta)_{\mathbb{R}^N}^{-\beta}\chi g\right)||_{L^p(\mathbb{R}^N)}+C||g||_{L^1(\Omega)}\\&\leq C||g||_{W^{1-2\beta,p}(\Omega)}
	\end{align*}
which implies \eqref{es54}.  The proof is complete.
\end{proof}


\section*{Comments and related problems}\label{sec.comm}

\noindent $\bullet$ We could do the same program with the spectral Laplacian replaced
by the other standard option, the so-called natural or restricted Laplacian on bounded domains. Other more  general integro-differential operators could also be considered.

\noindent $\bullet$ Concerning  similar problems posed on bounded domains, there is much recent work for porous medium equations involving nonlocal fractional operators  in the case of the model equation
\begin{equation}\label{eq.fpme}
 \partial_t u+(-\Delta)^{s} (F(u))=0
\end{equation}
usually for $F(u)=cu^m$, $m>0$. This includes the references  \cite{BV2014, BV2015, BSiV2015, BV.Bdd2016, BFV2016, BFR2017}. Higher regularity is treated in \cite{vpqrJEMS} and \cite{BFR2017}. The linear case is treated in \cite{BSiV2016}, and a case with $m<0$ in \cite{BSgV2016}.

As in the just mentioned model, we also want to address a number of questions. Our present model seems to be more difficult to analyze.

\noindent $\bullet$ There is a very important question of uniqueness for our model that cold be solved in one space dimension by using the viscosity ideas of \cite{BKM10}.

\noindent $\bullet$ Questions of regularity that must be proved: $C^{\alpha} $ regularity, higher regularity. Also the question of potential estimates.

\noindent $\bullet$ Question of finite speed of propagation, cf. works \cite{StTsVz15, StTsVz16, StTsVz17} for problems posed in the whole space. Regularity of free boundary problems, with open questions even  for PME with nonlocal pressure, \cite{CV1}.

\noindent $\bullet$ Questions of asymptotic behaviour, cf. the work \cite{BSiV2015} for equation \eqref{eq.fpme}.

\noindent $\bullet$ An interesting case in which a related problem is treated in a bounded domain concerns the work of Serfaty et al. \cite{AMS, AmSr08, SerVaz} on equations of superconductivity, which formally corresponds to $m_1=m_2=1$ with $s=1$.

 \noindent $\bullet$  In order to study the Cauchy problem in $\mR^N$, we may use as approximations the problems posed in a sequence of balls $B_{R_n}(0)$. Using the previous results in bounded domains (Theorems \ref{mainthm}, \ref{pro-Dis-data}, \ref{pro-measuredata}), and passing then to the limit $R_n\to \infty$ we can obtain existence and estimates for the solutions of the same equation posed in the whole space  with bounded and integrable data, or with merely integrable data, or bounded Radon measures. This is to be compared with the previous results of \cite{BiImKa, StTsVz16, StTsVz17}. Note that in these references only one nonlinearity is considered at a time, the approach is different, and $f=0$. This proposal needs careful elaboration.


\

\noindent {\large \sc Acknowledgment}

\noindent  The first named author has been supported by  the Centro De Giorgi, Scuola Normale Superiore, Pisa, Italy. The Second named author  has been supported by the Spanish Project MTM2014-52240-P.

\newpage


\noindent {\sc Keywords.}  Nonlocal nonlinear parabolic equations,  fractional Laplacian on a bounded domain.

\medskip

\noindent{\sc Mathematics Subject Classification}. 35R11, 35K61, 35K65.

\end{document}